\documentclass{amsart}
\title{Equivariant Eilenberg-Watts theorem for module coalgebras}

\author[T. Shibata]{Taiki Shibata}
\address[T. Shibata]{Department of Applied Mathematics,
  Okayama University of Science \\
  1-1 Ridai-cho, Kita-ku Okayama-shi, Okayama 700-0005, Japan.}
\email{shibata@ous.ac.jp}
\author[K. Shimizu]{Kenichi Shimizu}
\address[K. Shimizu]{Department of Mathematical Sciences,
  Shibaura Institute of Technology \\
  307 Fukasaku, Minuma-ku, Saitama-shi, Saitama 337-8570, Japan.}
\email{kshimizu@shibaura-it.ac.jp}


\date{}

\usepackage{amsmath,amssymb}
\usepackage[setpagesize=false]{hyperref}
\usepackage{tikz}
\usepackage{tikz-cd}
\usepackage{mathrsfs}
\usepackage{amsthm}
\numberwithin{equation}{section}
\theoremstyle{plain}
\newtheorem{C}{}[section] 
\newtheorem{lemma}[C]{Lemma}
\newtheorem{claim}[C]{Claim}
\newtheorem{theorem}[C]{Theorem}
\newtheorem{corollary}[C]{Corollary}

\theoremstyle{definition}
\newtheorem{definition}[C]{Definition}
\theoremstyle{remark}
\newtheorem{remark}[C]{Remark}
\newtheorem{example}[C]{Example}

\newcommand{\id}{\mathrm{id}}
\newcommand{\op}{\mathrm{op}}
\newcommand{\cop}{\mathrm{cop}}

\newcommand{\bfk}{\Bbbk}

\newcommand{\Hom}{\mathrm{Hom}}

\newcommand{\Ker}{\mathrm{Ker}}

\newcommand{\Img}{\mathrm{Im}}

\newcommand{\unitobj}{\mathbf{1}}
\newcommand{\rev}{\mathrm{rev}}
\newcommand{\eval}{\mathrm{ev}}
\newcommand{\coev}{\mathrm{coev}}

\newcommand{\copow}{\otimes}

\usepackage{upgreek}
\newcommand{\coact}{\updelta}
\newcommand{\coactl}{\coact^{\ell}}
\newcommand{\coactr}{\coact^{r}}
\newcommand{\coten}{\mathbin{\square}}

\newcommand{\Mod}{\mathfrak{M}}
\newcommand{\Vect}{\mathbf{Vec}}

\newcommand{\lax}{\mathrm{lax}}
\newcommand{\strong}{\mathrm{strong}}

\begin{document}

\begin{abstract}
  For coalgebras $C$ and $D$, Takeuchi proved that the category of linear functors from $\mathfrak{M}^C$ to $\mathfrak{M}^D$ preserving small coproducts is equivalent to the category of $C$-$D$-bicomodules, where $\mathfrak{M}^C$ for a coalgebra $C$ means the category of right $C$-comodules. We formulate and prove an equivariant version of this result for module coalgebras over a bialgebra. As an application, for a bialgebra $H$, we establish an equivalence of the 2-category of a particular class of module categories over the monoidal category $\mathfrak{M}^H$ and the 2-category of a particular class of module categories over the monoidal category ${}_H\mathfrak{M}$ of left $H$-modules.
\end{abstract}

\maketitle

\section{Introduction}

The Eilenberg-Watts equivalence states that an additive functor between categories of modules over rings has a right adjoint if and only if it is isomorphic to the functor given by tensoring a bimodule.
The starting point of this paper is an analogous result for coalgebras given by Takeuchi \cite{MR472967} as a part of his Morita theory for coalgebras, nowadays called Morita-Takeuchi theory.
To recall his result, we introduce some notation:
Given coalgebras $C$ and $D$ over a field $\bfk$, we denote by ${}^C\Mod$, $\Mod^D$ and ${}^C\Mod^D$ the category of left $C$-modules, right $D$-comodules and $C$-$D$-bicomodules, respectively.
Given linear categories $\mathcal{A}$ and $\mathcal{B}$, we denote by $\mathscr{L}(\mathcal{A}, \mathcal{B})$ the category of left exact linear functors from $\mathcal{A}$ to $\mathcal{B}$ that preserving coproducts existing in $\mathcal{A}$.
The aforementioned result of Takeuchi states that the functor
\begin{equation*}
  {}^C\Mod^D \to \mathscr{L}(\Mod^C, \Mod^D),
  \quad M \mapsto (-) \coten_C M
\end{equation*}
is an equivalence of linear categories, where $\coten_C$ means the cotensor product over the coalgebra $C$ (see Subsection~\ref{subsec:takeuchi-equiv}).

In this paper, we give an `equivariant' version of Takeuchi's equivalence, which may be thought of as a dual version of the equivariant Eilenberg-Watts theorem established by Andruskiewitsch and Mombelli \cite[Proposition 1.23]{MR2331768} (see also \cite[Appendix A]{NSS} for an exposition).
Let $H$ be a bialgebra, and let $C$ be a left $H$-module coalgebra. Then there is a functor $\otimes : \Mod^H \times \Mod^C \to \Mod^C$, which makes $\Mod^C$ a left module category over the monoidal category $\Mod^H$ (see Section~\ref{sec:equivariant-EW}).
Let $D$ also be a left $H$-module coalgebra.
A {\em lax left $\Mod^H$-module functor} from $\Mod^C$ to $\Mod^D$ is a functor $F: \Mod^C \to \Mod^D$ together with a natural transformation
\begin{equation*}
  \xi_{X,M} : X \otimes F(M) \to F(X \otimes M) \quad (X \in \Mod^H, M \in \Mod^C)
\end{equation*}
satisfying certain axioms (we note that the structure morphism $\xi$ is not required to be invertible).
The main result of this paper (Theorem~\ref{thm:equivariant-EW}) is that Takeuchi's equivalence induces a category equivalence ${}^C_H\Mod^D \approx \mathscr{L}_{\Mod^H}^{\lax}(\Mod^C, \Mod^D)$, where the source is the category of $C$-$D$-bicomodules in the monoidal category ${}_H\Mod$ of left $H$-modules and the target is the category of lax left $\Mod^H$-module functors from $\Mod^C$ to $\Mod^D$ whose underlying functor belongs to $\mathscr{L}(\Mod^C, \Mod^D)$.
In addition to proving this result, we also demonstrate several applications in later sections.

This paper is organized as follows: In Section~\ref{sec:preliminaries}, we discuss the canonical action of the category $\Vect$ of vector spaces on a linear category admitting small coproducts, and provide some relevant lemmas. We then recall Takeuchi's equivalence. We also recall Pareigis' Eilenerg-Watts theorem in a monoidal category (Theorem~\ref{thm:EW}) and point out that Takeuchi's equivalence is obtained as a tricky application of Pareigis' theorem.

The main result, Theorem~\ref{thm:equivariant-EW}, is proved in Section \ref{sec:equivariant-EW}.
Subsequently, in Section~\ref{sec:applications-Hopf-YD}, we apply the main result to Hopf and Yetter-Drinfeld modules. An important observation is that if $H$ is a Hopf algebra, then the inclusion functor
\begin{equation*}
  i_{C, D}:
  \mathscr{L}_{\Mod^H}^{\strong}(\Mod^C, \Mod^D)
  \to \mathscr{L}_{\Mod^H}^{\lax}(\Mod^C, \Mod^D)
\end{equation*}
is an equivalence for all left $H$-module coalgebras $C$ and $D$, where the target category is the full subcategory of strong module functors (Theorem~\ref{thm:equivalence-iota}). We explain how this theorem yield the fundamental theorem for Hopf modules and an equivalence for the category of Yetter-Drinfeld modules.

In Section~\ref{sec:applications-bicat}, we examine our results from a bicategorical perspective. Specifically, we establish a 2-equivalence between a suitably defined 2-category of left $\Mod^H$-module categories and a suitably defined 2-category of left ${}_H\Mod$-module categories (Theorem~\ref{thm:duality-1}). We also establish a similar result for left $\Mod^H$-module categories and right ${}_H\Mod$-module categories (Theorem \ref{thm:duality-2}). Finally, we formulate an equivariant version of Morita-Takeuchi equivalence and give several characterizations of this equivalence relation (Theorem~\ref{thm:H-Morita-Takeuchi}).

In Section~\ref{sec:applications-mod-full-subcat}, we discuss module subcategories of $\Mod^C$. It is known that the set of closed subcategories of $\Mod^C$ are in bijection with the set of subcoalgebras of $C$ \cite[Theorem 2.5.5]{MR1786197}. We establish an equivariant version of this result by giving a bijection between the set of $H$-module subcoalgebras of $C$ and the set of $\Mod^H$-module closed subcategories of $\Mod^C$.

\subsection*{Acknowledgements}

The second author (T.S.) is supported by JSPS KAKENHI Grant Number JP22K13905.
The third author (K.S.) is supported by JSPS KAKENHI Grant Number JP24K06676

\section{Preliminaries}
\label{sec:preliminaries}

\subsection{Notation}

Throughout this paper, we work over a field $\bfk$.
We denote by $\Vect$ the category of all vector spaces over the field $\bfk$.
Unless otherwise noted, a (co)algebra means a (co)associative and (co)unital (co)algebra over $\bfk$.
Given two algebras $A$ and $B$, we denote by ${}_A \Mod$, $\Mod_B$ and ${}_A \Mod_B$ the category of left $A$-modules, right $B$-modules and $A$-$B$-bimodules, respectively.
Similarly, given two coalgebras $C$ and $D$, we denote by ${}^C\Mod$, $\Mod^D$ and ${}^C\Mod^D$ the category of left $C$-comodules, right $D$-comodules and $C$-$D$-bicomodules, respectively.
The Hom functors of these categories are denoted by the same symbol but with $\Mod$ replaced with $\Hom$.
For example, the Hom functor of ${}^C\Mod$ is denoted by ${}^C\Hom$.

For the basic theory of coalgebras, we refer the reader to \cite{MR1786197}.
The comultiplication and the counit of a coalgebra $C$ are denoted by $\Delta_C$ and $\varepsilon_C$, respectively, but the subscript $C$ is often dropped when it is clear from the context. To express the comultiplication, we use Sweedler's notation, such as
\begin{equation*}
  \Delta(c) = c_{(1)} \otimes c_{(2)}, \quad
  \Delta(c_{(1)}) \otimes c_{(2)} = c_{(1)} \otimes c_{(2)} \otimes c_{(3)}
  = c_{(1)} \otimes \Delta(c_{(2)})
\end{equation*}
for $c \in C$. Given $M \in \Mod^C$ and $N \in {}^C\Mod$, we denote by $\coactr_M$ and $\coactl_N$ the coaction of $C$ on $M$ and $N$, respectively.
We also use Sweedler's notation, such as
\begin{equation*}
  \coactr_M(m) = m_{(0)} \otimes m_{(1)}
  \quad (m \in M)
  \quad \text{and} \quad
  \coactl_N(n) = n_{(-1)} \otimes n_{(0)}
  \quad (n \in N),
\end{equation*}
to express the coaction of $C$. The {\em contensor product} of $M$ and $N$ over $C$ \cite[Section 2.3]{MR1786197} is defined and denoted by
\begin{equation}
  \label{eq:def-cotensor-product}
  M \coten_C N := \Ker(\coactr_M \otimes \id_N - \id_M \otimes \coactl_N : M \otimes N \to M \otimes C \otimes N).
\end{equation}

\subsection{Modules over a monoidal category}

We refer the reader to \cite{MR3242743} for basics on monoidal categories and module categories over a monoidal category. All monoidal categories are assumed to be strict in view of Mac Lane's strictness theorem.
All module categories are also assumed to be strict in the light of \cite[Remark 7.2.4]{MR3242743}.
Unless otherwise noted, we denote by $\otimes$ and $\unitobj$ the monoidal product and the unit object of a monoidal category, respectively. The symbol $\otimes$ is also used to express the action of a monoidal category on a module category.

A module functor \cite[Definition 7.2.1]{MR3242743} is a functor between module categories over the same monoidal category, say $\mathcal{C}$, that preserves the action of $\mathcal{C}$ up to coherent natural isomorphism.
In this paper, we also need a relaxed notion: Let $\mathcal{M}$ and $\mathcal{N}$ be left $\mathcal{C}$-module categories. Then a {\em lax left $\mathcal{C}$-module functor} from $\mathcal{M}$ to $\mathcal{N}$ is a functor $F: \mathcal{M} \to \mathcal{N}$ equipped with a natural transformation
\begin{equation*}
  \xi_{X,M} : X \otimes F(M) \to F(X \otimes M)
  \quad (X \in \mathcal{C}, M \in \mathcal{M}),
\end{equation*}
which is not necessarily invertible, such that the equations
\begin{equation}
  \label{eq:def-lax-module-functor}
  \xi_{X \otimes Y, M}
  = \xi_{X, Y \otimes M} \circ (\id_X \otimes \xi_{Y,M})
  \quad \text{and} \quad
  \xi_{\unitobj, M} = \id_{F(M)}
\end{equation}
hold for all objects $M \in \mathcal{M}$ and $X, Y \in \mathcal{C}$.
Morphisms of lax right $\mathcal{C}$-module functors are defined in the same way as \cite[Definition 7.2.3]{MR3242743}.

A {\em strong $\mathcal{C}$-module functor} is a lax $\mathcal{C}$-module functor whose structure morphism is invertible.
In the case where $\mathcal{C}$ is a rigid monoidal category, every lax $\mathcal{C}$-module functor is strong \cite{MR3934626}.
However, we deal with non-rigid monoidal categories and cannot avoid considering lax module functors in this paper.

By an {\em equivalence} of left $\mathcal{C}$-module categories, we mean a lax left $\mathcal{C}$-module functor $F : \mathcal{M} \to \mathcal{N}$ for which there exists a lax left $\mathcal{C}$-module functor $G: \mathcal{N} \to \mathcal{M}$ such that both $G \circ F \cong \id_{\mathcal{M}}$ and $F \circ G \cong \id_{\mathcal{N}}$ as lax left $\mathcal{C}$-module functors. It is easy to see that an equivalence of left $\mathcal{C}$-module categories is the same thing as a strong left $\mathcal{C}$-module functor whose underlying functor is an equivalence of categories.

\subsection{Canonical $\Vect$-action}
\label{subsec:can-vec-action}

There is a close relationship between categories on which a monoidal category $\mathcal{C}$ acts and categories enriched by $\mathcal{C}$; see, {\it e.g.}, \cite{MR1466618}, where a more general setting that $\mathcal{C}$ is a bicategory is discussed. Here we give some remarks on the action of $\Vect$ on a linear category.

A linear category is nothing but a category enriched by $\Vect$.
For an object $M$ of a linear category $\mathcal{M}$ and a vector space $X$, we define the copower $X \copow M$ of $M$ by $X$ to be an object representing the linear functor
\begin{equation*}
  \Hom_{\bfk}(X, \Hom_{\mathcal{M}}(M, -)) : \mathcal{M} \to \Vect.
\end{equation*}
Let $\alpha$ be the dimension of $X$ over $\bfk$. Then there is an isomorphism
\begin{equation*}
  \Hom_{\bfk}(X, \Hom_{\mathcal{M}}(M, -))
  \cong \prod_{\alpha} \Hom_{\mathcal{M}}(M, -)
\end{equation*}
given by using a basis of $X$. Since the coproduct $M^{\oplus \alpha}$ is an object representing the right hand side, the copower $X \copow M$ exists if and only if $M^{\oplus \alpha}$ does. Moreover, if they exist, there is an isomorphism $X \copow M \cong M^{\oplus \alpha}$. This isomorphism is natural in $M$, however, depends on the choice of a basis of $X$.

Now we assume that $\mathcal{M}$ admits small coproducts so that the copower $X \copow M$ exists for all objects $X \in \Vect$ and $M \in \mathcal{M}$.
The assignment $(X, M) \mapsto X \copow M$ extends to a bilinear functor $\Vect \times \mathcal{M} \to \mathcal{M}$, which we call the {\em canonical action} of $\Vect$ as it makes $\mathcal{M}$ a module category over $\Vect$.

From now on, a linear category admitting small coproducts is viewed as a module category over $\Vect$ by the canonical action. By definition, the functor $(-) \copow M : \Vect \to \mathcal{M}$ is left adjoint to $\Hom_{\mathcal{M}}(M, -) : \mathcal{M} \to \Vect$. We denote by
\begin{gather*}
  \coev_{M,X} : X \to \Hom_{\mathcal{M}}(M, X \copow M)
  \quad (X \in \Vect), \\
  \eval_{M,M'} : \Hom_{\mathcal{M}}(M, M') \copow M \to M'
  \quad (M' \in \mathcal{M})
\end{gather*}
the unit and the counit of the adjunction, respectively. For an element $v$ of a vector space $V$, we denote by $\underline{v} : \bfk \to V$ the unique linear map sending $1_{\bfk}$ to $v$. Under the identification $\bfk \copow M = M$, we have
\begin{equation}
  \label{eq:canonical-Vect-ev-coev}
  \eval_{M, M'} \circ (\underline{f} \otimes \id_M) = f
  \quad \text{and} \quad
  \coev_{M,X}(x) = \underline{x} \otimes \id_M
\end{equation}
for all $f \in \Hom_{\mathcal{M}}(M, M')$ and $x \in X \in \Vect$.

\subsection{Canonical $\Vect$-module structure}
\label{subsec:can-vec-module-structure}

Let $\mathcal{M}$ and $\mathcal{N}$ be linear categories admitting small coproducts. For a linear functor $F: \mathcal{M} \to \mathcal{N}$, we define the morphism
\begin{equation}
  \label{eq:canonical-Vect}
  \widehat{F}_{X,M} : X \copow F(M) \to F(X \copow M)
  \quad (X \in \Vect, M \in \mathcal{M})
\end{equation}
to be the element corresponding to the linear map
\begin{equation}
  \label{eq:canonical-Vect-def-1}
  X \to \Hom_{\mathcal{N}}(F(M), F(X \copow M)),
  \quad x \mapsto F(\underline{x} \copow \id_M)
  \quad (x \in X)
\end{equation}
under the adjunction isomorphism
\begin{equation}
  \label{eq:canonical-Vect-def-2}
  \Hom_{\bfk}(X, \Hom_{\mathcal{N}}(F(M), F(X \copow M)))
  \cong \Hom_{\mathcal{N}}(X \copow F(M), F(X \copow M)).
\end{equation}
We call the family $\widehat{F} = \{ \widehat{F}_{X,M} \}$ of morphisms {\em the canonical $\Vect$-module structure} of $F$ as it makes $F$ a lax left $\Vect$-module functor. The construction $F \mapsto (F, \widehat{F})$ is functorial. Namely, if $\alpha : F \to G$ is a natural transformation between linear functors $F$ and $G$ from $\mathcal{M}$ to $\mathcal{N}$, then $\alpha : (F, \widehat{F}) \to (G, \widehat{G})$ is a morphism of lax $\Vect$-module functors.

\begin{example}
  Let $C$ be a coalgebra.
  The copower of $M \in \Mod^C$ by $X \in \Vect$ is the tensor product $X \otimes M$ over $\bfk$ equipped with the right $C$-coaction given by
  \begin{equation*}
    X \otimes M \to X \otimes M \otimes C,
    \quad x \otimes m \mapsto x \otimes m_{(0)} \otimes m_{(1)}
    \quad (m \in M, x \in X).
  \end{equation*}
  The canonical $\Vect$-module structure of $\Hom^C(M, -) : \Mod^C \to \Vect$ is given by
  \begin{equation*}
    X \otimes \Hom^C(M, W) \to \Hom^C(M, X \otimes W),
    \quad x \otimes f \mapsto (m \mapsto x \otimes f(m)).
  \end{equation*}
  As this example shows, the canonical lax $\Vect$-module structure of a linear functor is not necessarily invertible.
\end{example}

\begin{example}
  \label{ex:canonical-Vect-module-structure}
  For $X \in \Vect$, $M \in \Mod^C$ and $N \in {}^C\Mod$, one has
  \begin{equation*}
    (X \otimes M) \coten_C N = X \otimes (M \coten_C N)
  \end{equation*}
  as subspaces of $X \otimes M \otimes N$.
  The canonical $\Vect$-module structure of the cotensor functor $(-) \coten_C N$ is the identity natural transformation.
\end{example}

A linear functor is essentially the same thing as a lax $\Vect$-module functor with respect to the canonical action of $\Vect$. The following lemma is a mathematical formulation of this slogan:

\begin{lemma}
  \label{lem:Vec-module-vs-Vec-enriched}
  Let $\mathcal{M}$ and $\mathcal{N}$ be linear categories admitting small coproducts. Then the category of linear functors from $\mathcal{M}$ to $\mathcal{N}$ is isomorphic to the category of lax $\Vect$-module functors from $\mathcal{M}$ to $\mathcal{N}$.
\end{lemma}

This lemma is a very special case of \cite[Theorem 3.7]{MR1466618}, where they have established an equivalence between a certain class of module categories over a bicategory $\mathfrak{W}$ and a certain class of categories enriched by $\mathfrak{W}$.
For reader's convenience, we provide a direct proof.

\begin{proof}
  We have already established a fully faithful functor $F \mapsto (F, \widehat{F})$ from the category of linear functors from $\mathcal{M}$ to $\mathcal{N}$, to the category of lax $\Vect$-module functors from $\mathcal{M}$ to $\mathcal{N}$. The functor $F \mapsto (F, \widehat{F})$ is clearly injective on objects.
  We shall verify the surjectivity on objects. Let $(F, \xi)$ be an lax $\Vect$-module functor from $\mathcal{M}$ to $\mathcal{N}$.
  We first show that $F$ is in fact a linear functor.
  To see this, we fix two objects $M, M' \in \mathcal{M}$, write $H = \Hom_{\mathcal{M}}(M, M')$ for short, and define
  \begin{equation*}
    F_{M,M'} \in \Hom_{\bfk}(H, \Hom_{\mathcal{N}}(F(M), F(M')))
  \end{equation*}
  to be the element corresponding to
  \begin{equation*}
    F_{M, M'}^{\sharp} := F(\eval_{M, M'}) \circ \xi_{H,M} \in \Hom_{\mathcal{N}}(H \copow F(M), F(M')).
  \end{equation*}
  Explicitly, $F_{M,M'} = \Hom_{\bfk}(F(M), F_{M,M'}^{\sharp}) \circ \coev_{H, F(M)}$.
  For all $f \in H$, we have
  \begin{gather*}
    F_{M,M'}(f)
    = F^{\sharp}_{M,M'} \circ \coev_{M,H}(f)
    \mathop{=}^{\eqref{eq:canonical-Vect-ev-coev}}
    F(\eval_{M,M'}) \circ \xi_{H,M} \circ (\underline{f} \copow \id_{F(M)}) \\
    = F(\eval_{M,M'}) \circ F(\underline{f} \copow \id_M) \circ \xi_{\bfk, M}
    \mathop{=}^{\eqref{eq:canonical-Vect-ev-coev}}
    F(f).
  \end{gather*}
  Namely, the map $\Hom_{\mathcal{M}}(M, M') \to \Hom_{\mathcal{M}}(F(M), F(M'))$ induced by $F$ is identical to the linear map $F_{M,M'}$. Therefore $F$ is linear.

  It remains to show $\xi = \widehat{F}$. For $x \in X \in \Vect$ and $M \in \mathcal{M}$, we have
  \begin{gather*}
    \xi_{X,M} \circ (\underline{x} \otimes \id_{F(M)})
    = F(\underline{x} \otimes \id_M) \circ \xi_{\bfk, M}
    = F(\underline{x} \copow \id_M).
  \end{gather*}
  Thus, by~\eqref{eq:canonical-Vect-ev-coev}, we have $\xi_{X,M} \circ \coev_{F(M), X}(x) = F(\coev_{M,X}(x))$ for all $x \in X$.
  Under the canonical isomorphism~\eqref{eq:canonical-Vect-def-2}, the left and the right hand side of this equation correspond to $\xi_{X,M}$ and $\widehat{F}_{X,M}$, respectively. Hence $\xi_{X,M} = \widehat{F}_{X,M}$. The proof is done.
\end{proof}

We discuss when the canonical lax $\Vect$-module structure is invertible.
For this purpose, we give the following alternative description of the canonical $\Vect$-module structure:
Let $\mathcal{M}$ and $\mathcal{N}$ be linear categories admitting small coproducts, and let $F: \mathcal{M} \to \mathcal{N}$ be a linear functor. For an object $M \in \mathcal{M}$ and a vector space $X$ of dimension $\alpha$, there is a morphism
\begin{equation}
  \label{eq:canonical-Vect-2}
  \begin{tikzcd}
    X \copow F(M) \ar[r, "{\cong}"]
    & F(M)^{\oplus \alpha} \ar[r]
    & F(M^{\oplus \alpha}) \ar[r, "{\cong}"]
    & F(X \copow M),
  \end{tikzcd}
\end{equation}
where the first and the third arrows are isomorphisms given by a basis of $X$ and the second arrow is the canonical morphism obtained by the universal property of the coproduct.

\begin{lemma}
  \label{lem:canonical-Vect-2}
  The morphism \eqref{eq:canonical-Vect-2} is equal to $\widehat{F}_{X,M}$.
\end{lemma}
\begin{proof}
  We let $f_{X,M}$ be the morphism \eqref{eq:canonical-Vect-2}.
  Let $\{ x_{i} \}_{i \in I}$ be the basis of $X$ used in \eqref{eq:canonical-Vect-2}, where $I$ is a set of cardinality $\alpha$. Let, in general, $L$ be an object of a linear category admitting small coproducts. If we identify the copower $X \copow L$ with $L^{\oplus I}$ by using the basis $\{ x_i \}_{i \in I}$, the $i$-th inclusion $L \to L^{\oplus I}$ is identified with $\underline{x_i} \copow \id_L : L \to X \copow L$. Hence $f_{X,M}$ is characterized by the following equation:
  \begin{equation*}
    f_{X, M} \circ (\underline{x_i} \copow \id_{F(M)})
    = F(\underline{x_i} \otimes \id_M)
    \quad (i \in I).
  \end{equation*}
  By~\eqref{eq:canonical-Vect-ev-coev} and linearity, this equation is equivalent to
  \begin{equation*}
    f_{X, M} \circ \coev_{F(M),X}(x)
    = F(\coev_{M,X}(x))
    \quad (x \in X).
  \end{equation*}  
  This means that $f_{X,M}$ corresponds to the linear map \eqref{eq:canonical-Vect-def-1} under the canonical isomorphism \eqref{eq:canonical-Vect-def-2}. Thus we have $f_{X,M} = \widehat{F}_{X,M}$. The proof is done.
\end{proof}

By the above lemma, $\widehat{F}$ is invertible if and only if $F$ preserves all coproducts of the form $M^{\oplus \alpha}$ for some object $M \in \mathcal{M}$ and cardinal $\alpha$. The following fact, used by Takeuchi in \cite[Proposition 2.1]{MR472967}, is now obvious:

\begin{lemma}
  \label{lem:coproduct-preserving-implies-strong-Vec}
  Let $F: \mathcal{M} \to \mathcal{N}$ be a linear functor between linear categories $\mathcal{M}$ and $\mathcal{N}$ admitting small coproducts. If $F$ preserves small coproducts, then the natural transformation $\widehat{F}$ is invertible.
\end{lemma}

It is interesting to know whether the converse of this lemma holds. We do not know a general answer to this question, however, we can give an affirmative answer under some conditions on $\mathcal{M}$, $\mathcal{N}$ and $F$ as follows:

\begin{lemma}
  \label{lem:coproduct-preserving-implies-strong-Vec-2}
  Let $F : \mathcal{M} \to \mathcal{N}$ be as in Lemma \ref{lem:coproduct-preserving-implies-strong-Vec}, and assume that $\widehat{F}$ is invertible. Then $F$ preserves small coproducts if one of the following conditions hold.
  \begin{enumerate}
  \item Both $\mathcal{M}$ and $\mathcal{N}$ are abelian categories, and $F$ is right exact.
  \item Both $\mathcal{M}$ and $\mathcal{N}$ are Ab4 abelian categories, and $F$ is left exact.
  \end{enumerate}
\end{lemma}

We recall that an abelian category $\mathcal{A}$ satisfies Ab4 if $\mathcal{A}$ admits small coproducts and coproducts are exact in $\mathcal{A}$.

\begin{proof}
  We assume (1).
  Let $\{ M_i \}_{i \in I}$ be a family of objects of $\mathcal{M}$, and set $M = \bigoplus_{i \in I} M_i$.
  For each element $i \in I$, we consider the exact seqeunce $M \to M \to M_i \to 0$ in $\mathcal{M}$, where the first and the second arrow are the projection to $\bigoplus_{j \in I \setminus \{ i \}} M_j$ and $M_i$, respectively. Then we have the commutative diagram
  \begin{equation*}
    \begin{tikzcd}
      F(M)^{\oplus I} \arrow[r] \arrow[d, "\xi"]
      & F(M)^{\oplus I} \arrow[r] \arrow[d, "\xi"]
      & \bigoplus_{i \in I} F(M_i) \arrow[r] \arrow[d, "\xi'"]
      & 0 \\
      F(M^{\oplus I}) \arrow[r]
      & F(M^{\oplus I}) \arrow[r]
      & F(\bigoplus_{i \in I} M_i) \arrow[r]
      & 0,
    \end{tikzcd}
  \end{equation*}
  where $\xi$ and $\xi'$ are canonical morphisms obtained by the universal property of the coproducts.
  Since $F$ and the coproduct are right exact, the rows of the diagram are exact.
  By Lemma~\ref{lem:canonical-Vect-2} and the assumption that $\widehat{F}$ is an isomorphism, $\xi$ is an isomorphism. 
  Thus $\xi'$ is also an isomorphism. Namely, $F$ preserves coproducts.

  In the case of (2), we consider an exact sequence of the form $0 \to M_i \to M \to M$ instead of the exact sequence $M \to M \to M_i \to 0$. By the left exactness of $F$ and the Ab4 assumption, one can prove that $F$ preserves coproducts in the case of (2) in a similar way as the case of (1).
\end{proof}

\subsection{Takeuchi's equivalence}
\label{subsec:takeuchi-equiv}

We introduce the following terminology and notation as it is the class of functors of central interest in this paper:

\begin{definition}
  Given cocomplete abelian linear categories $\mathcal{A}$ and $\mathcal{B}$, we denote by $\mathscr{L}(\mathcal{A}, \mathcal{B})$ the category of left exact linear functors from $\mathcal{A}$ to $\mathcal{B}$ preserving coproducts. A functor $\mathcal{A} \to \mathcal{B}$ is said to be {\em of type $\mathscr{L}$} if it belongs to $\mathscr{L}(\mathcal{A}, \mathcal{B})$.
\end{definition}

As we have mentioned in Introduction, the following theorem holds:

\begin{theorem}[Takeuchi {\cite{MR472967}}]
  \label{thm:EW-coalgebra}
  For coalgebras $C$ and $D$, the functor
  \begin{equation*}
    \Phi_{C,D} : {}^C\Mod^D \to \mathscr{L}(\Mod^C, \Mod^D),
    \quad M \mapsto (-) \coten_C M
  \end{equation*}
  is an equivalence of linear categories.
\end{theorem}

Lemma~\ref{lem:coproduct-preserving-implies-strong-Vec} is essential for the construction of the quasi-inverse of $\Phi_{C,D}$. Indeed, it is constructed in the proof of \cite[Proposition 2.1]{MR472967} as follows: Let $F$ be an object of $\mathscr{L}(\Mod^C, \Mod^D)$. Lemma~\ref{lem:coproduct-preserving-implies-strong-Vec} says that the canonical $\Vect$-module structure $\widehat{F}$ is invertible.
A quasi-inverse of $\Phi_{C,D}$ is given by $\Phi_{C,D}^{-1}(F) = F(C)$, where the right $D$-comodule $F(C)$ is made into a $C$-$D$-bicomodule by the left $C$-coaction
\begin{equation*}
  F(C) \xrightarrow{\quad F(\Delta_C) \quad}
  F(C \copow C)
  \xrightarrow{\quad \widehat{F}_{C,C}^{-1} \quad}
  C \copow F(C).
\end{equation*}

\subsection{Eilenberg-Watts equivalence}

By an {\em abelian monoidal category}, we mean an abelian category endowed with a structure of a monoidal category such that the tensor product is additive in each variable.
Let $\mathcal{C}$ be an abelian monoidal category whose tensor product functor is right exact in each variable. Given algebras $A$ and $B$ in $\mathcal{C}$ ($=$ monoids in $\mathcal{C}$ \cite{MR0498792,MR1712872}), we denote by ${}_A\mathcal{C}$, $\mathcal{C}_B$ and ${}_A\mathcal{C}_B$ the category of left $A$-modules, right $B$-modules and $A$-$B$-bimodules in $\mathcal{C}$, respectively. For $M \in \mathcal{C}_A$ and $N \in {}_A\mathcal{C}$, their tensor product over $A$ is defined and denoted by
\begin{equation*}
  M \otimes_A N = \text{Coker}(a^{r} \otimes \id_N - \id_M \otimes a^{\ell} : M \otimes A \otimes N \to M \otimes N),
\end{equation*}
where $a^{r} : M \otimes A \to M$ and $a^{\ell} : A \otimes N \to N$ are the left and the right action of $A$ on $M$ and $N$, respectively. An object $M \in {}_A\mathcal{C}_B$ defines a right exact additive functor $(-) \otimes_A M : \mathcal{C}_A \to \mathcal{C}_B$.

The Eilenberg-Watts theorem asserts that a cocontinuous linear functor between categories of modules over an algebra is given by tensoring a bimodule.
We note that $\mathcal{C}_R$ for an algebra $R$ in $\mathcal{C}$ is a left $\mathcal{C}$-module category. Moreover, by our assumption that the tensor product of $\mathcal{C}$ is right exact in each variable, we see that the functor $(-) \otimes_A M$ as above is a strong left $\mathcal{C}$-module functor. Strong module functors of this form are characterized as follows:

\begin{theorem}[{Pareigis \cite[Theorem 4.2]{MR0498792}}]
  \label{thm:EW}
  Let $\mathcal{C}$ be an abelian monoidal category with right exact monoidal product.
  For algebras $A$ and $B$ in $\mathcal{C}$, the following fucntor is a category equivalence:
  \begin{equation}
    \label{eq:EW-theorem-Pareigis}
    \begin{aligned}
      {}_A\mathcal{C}_B
      & \to \{ \text{right exact strong $\mathcal{C}$-module functors $\mathcal{C}_A \to \mathcal{C}_B$} \}, \\
      M & \mapsto (-) \otimes_A M.
    \end{aligned}
  \end{equation}
\end{theorem}

One may wonder why the cocontinuous condition in the original Eilenberg-Watts theorem does not appear explicitly in this theorem. This issue can be explained by the discussion in Subsection~\ref{subsec:can-vec-module-structure} as follows: Since any colimit in an abelian category can be built from coproducts and coequalizers, a right exact linear functor between categories of modules over algebras is cocontinuous if and only if it preserves coproducts.
Thus, when $\mathcal{C} = \Vect$, the target category of \eqref{eq:EW-theorem-Pareigis} is identified with the category of cocontinuous linear functors from $\mathcal{C}_A$ to $\mathcal{C}_B$ by Lemmas~\ref{lem:coproduct-preserving-implies-strong-Vec} and \ref{lem:coproduct-preserving-implies-strong-Vec-2}.

Our consideration in Subsection~\ref{subsec:can-vec-module-structure} also clarifies the relation between Takeuchi's Theorem \ref{thm:EW-coalgebra} and Pareigis' Theorem~\ref{thm:EW}. In fact, Theorem~\ref{thm:EW-coalgebra} is a consequence of Theorem~\ref{thm:EW}. To explain this, we consider the case where $\mathcal{C} = \mathcal{T}^{\op}$ for some abelian monoidal category $\mathcal{T}$ with left exact tensor product.
Let $C$ and $D$ be coalgebras in the category $\mathcal{T}$.
If we denote by $A$ and $B$ the coalgebra $C$ and $D$ viewed as an algebra in $\mathcal{C}$, respectively, then the category ${}_A\mathcal{C}_B$ is the opposite category of the category of $C$-$D$-bicomodules in $\mathcal{T}$. Moreover, the tensor product over $A$ is the same thing as the cotensor product over $C$ (which is defined by the same formula as~\eqref{eq:def-cotensor-product} in a monoidal category). Thus, by applying Theorem~\ref{thm:EW} to $\mathcal{C} = \mathcal{T}^{\op}$, we obtain:

\begin{corollary}
  \label{cor:EW-Pareigis}
  Let $\mathcal{T}$ be an abelian monoidal category with left exact monoidal product.
  For coalgebras $C$ and $D$ in $\mathcal{T}$, the assignment $M \mapsto (-) \coten_C M$ establishes an equivalence from the category of $C$-$D$-bicomodules in $\mathcal{T}$ to the category of left exact strong $\mathcal{T}$-module functors from the category of right $C$-comodules in $\mathcal{T}$ and that of right $D$-comodules in $\mathcal{T}$.
\end{corollary}

Now let $C$ and $D$ be ordinary coalgebras. By Lemmas~\ref{lem:Vec-module-vs-Vec-enriched}--\ref{lem:coproduct-preserving-implies-strong-Vec-2}, the category $\mathscr{L}(\Mod^C, \Mod^D)$ in Theorem \ref{thm:EW-coalgebra} is identified with the category of left exact strong $\Vect$-module functors from $\Mod^C$ to $\Mod^D$. Thus Theorem \ref{thm:EW-coalgebra} follows from the above corollary applied to $\mathcal{T} = \Vect$.

\section{Equivariant Eilenberg-Watts theorem}
\label{sec:equivariant-EW}

\subsection{Main result}

Let $H$ be a bialgebra, and let $C$ be a left $H$-module coalgebra, that is, a coalgebra in the monoidal category ${}_H\Mod$ of left $H$-modules. If $X \in \Mod^H$ and $M \in \Mod^C$, then the vector space $X \otimes M$ becomes a right $C$-comodule by the coaction given by
\begin{equation*}
  \coactr_{X \otimes M}(x \otimes m) = x_{(0)} \otimes m_{(0)} \otimes x_{(1)} m_{(1)}
\end{equation*}
for $x \in X$ and $m \in M$. As is well-known, the category $\Mod^C$ is a left module category over $\Mod^H$ by this operation.

Let $D$ also be a left $H$-module coalgebra. We denote by ${}^C_H\Mod^D$ the category of $C$-$D$-bicomodules in ${}_H\Mod$. Namely, an object of this category is a left $H$-module $M$ equipped with a structure of $C$-$D$-bicomodule such that the equations
\begin{equation}
  \label{eq:C-H-Mod-D-axiom}
  \coactl_M(h m) = h_{(1)} m_{(-1)} \otimes h_{(2)} m_{(0)}, \quad
  \coactr_M(h m) = h_{(1)} m_{(0)} \otimes h_{(2)} m_{(1)}
\end{equation}
hold for all elements $h \in H$ and $m \in M$. A morphism in ${}^C_H\Mod^D$ is a linear map that is left $H$-linear, left $C$-colinear and right $D$-colinear.

In this section, we state and prove the main result of this paper, which says that the category of lax $\Mod^H$-module linear functors from $\Mod^C$ to $\Mod^D$ preserving coproducts is equivalent to ${}^C_H\Mod^D_{}$. We first give a functor from the latter category to the former.
Given $M \in {}^C\Mod^D$, we write
\begin{equation*}
  T_M := (-) \coten_C M : \Mod^C \to \Mod^D.
\end{equation*}

\begin{lemma}
  \label{lem:equivariant-EW-1}
  Suppose that $M$ is an object of ${}^C_H\Mod^D_{}$ with left $H$-action $\bullet$.
  Then $T_M$ becomes a lax right $\Mod^H$-module functor together with the structure map
  \begin{equation*}
    \xi^{\bullet}_{X,W} : X \otimes T_M(W) \to T_M(X \otimes W)
    \quad (W \in \Mod^C, X \in \Mod^H)
  \end{equation*}
  induced by the linear map
  \begin{equation}
    \label{eq:T-M-structure-morphism}
    \begin{aligned}
      X \otimes (W \otimes M)
      & \to (X \otimes W) \otimes M, \\
      x \otimes (w \otimes m)
      & \mapsto (x_{(0)} \otimes w) \otimes (x_{(1)} \bullet m)
      \quad (x \in X, w \in W, m \in M)
    \end{aligned}
  \end{equation}
\end{lemma}

Before we give a proof of this lemma, we introduce the inverse of the construction given in this lemma. We fix a bicomodule $M \in {}^C\Mod^D$ and assume that $T_M : \Mod^C \to \Mod^D$ is a lax $\Mod^H$-module functor together with structure morphism
\begin{equation*}
  \xi_{X,W} : X \otimes T_M(W) = X \otimes (W \coten_C M)
  \to (X \otimes W) \coten_C M = T_M(X \otimes W),
\end{equation*}
where $X \in \Mod^H$ and $W \in \Mod^C$.
Since $C \coten_C M = \Img(\coactl_M)$, the element $h \otimes m_{(-1)} \otimes m_{(0)}$ for $h \in H$ and $m \in M$ belongs to the domain of the linear map $\xi_{H,C}$. Now we define the bilinear operation $\bullet_{\xi} : H \times M \to M$ by
\begin{equation}
  \label{eq:equivariant-EW-1-action-def}
  h \bullet_{\xi} m := (\varepsilon_H \otimes \varepsilon_C \otimes \id_M)
  \xi_{H, C}(h \otimes m_{(-1)} \otimes m_{(0)})
\end{equation}
for $h \in H$ and $m \in M$. Then we have:

\begin{lemma}
  \label{lem:equivariant-EW-2}
  The bicomodule $M$ is an object of ${}_H^C\Mod^D$ by the left $H$-action $\bullet_{\xi}$.
\end{lemma}

We postpone the proof of this lemma. Based on Lemmas \ref{lem:equivariant-EW-1} and \ref{lem:equivariant-EW-2}, the main result of this paper is stated as follows:

\begin{theorem}
  \label{thm:equivariant-EW}
  Let $H$, $C$ and $D$ be as above. For each $C$-$D$-bicomodule $M$, the maps $\bullet \mapsto \xi^{\bullet}$ and $\xi \mapsto \bullet_{\xi}$ of Lemmas \ref{lem:equivariant-EW-1} and \ref{lem:equivariant-EW-2} establish bijections between the following two sets:
  \begin{enumerate}
  \item[(1)] The set of maps $H \times M \to M$ making $M$ an object of ${}^C_H\Mod^D$.
  \item[(2)] The set of natural transformations making the functor $T_M$ a lax left $\Mod^H$-module functor from $\Mod^C$ to $\Mod^D$.
  \end{enumerate}
  Furthermore, Takeuchi's equivalence induces an equivalence
  \begin{equation*}
    {}_H^C\Mod^D \approx \mathscr{L}_{\Mod^H}^{\lax}(\Mod^C, \Mod^D),
  \end{equation*}
  where the right hand side is the category of lax $\Mod^H$-module functors from $\Mod^C$ to $\Mod^D$ whose underlying functor is of type $\mathscr{L}$.
\end{theorem}

One can formulate and prove analogous results for right module coalgebras and bimodule coalgebras in the same way as this theorem. The detail is given in Subsection~\ref{subsec:equivariant-EW-variants}.

\subsection{Proof of Theorem~\ref{thm:equivariant-EW}}
We begin the proof of Theorem~\ref{thm:equivariant-EW}.
First of all, we shall establish yet unverified Lemmas~\ref{lem:equivariant-EW-1} and \ref{lem:equivariant-EW-2}. We recall the setting of the former lemma: $M$ is an object of the category ${}^C_H\Mod^D_{}$ with left $H$-action $\bullet$, and $\xi^{\bullet}_{X,W}$ for $X \in \Mod^H$ and $W \in \Mod^C$ is the linear map induced by \eqref{eq:T-M-structure-morphism}.

\begin{proof}[Proof of Lemma~\ref{lem:equivariant-EW-1}]
  We first check that the linear map $\xi^{\bullet}_{X,W}$ is well-defined.
  Let $x \in X$ and $t \in W \coten_C M$. Although $t$ is not a simple tensor in general, we write $t = w \otimes m$ for simplicity of notation. We note that the equation $w_{(0)} \otimes w_{(1)} \otimes m = w \otimes m_{(-1)} \otimes m_{(0)}$ holds since $t$ belongs to the cotensor product. Thus we have
  \begin{align*}
    & (x_{(0)} \otimes w) \otimes (x_{(1)} \bullet m)_{(-1)} \otimes (x_{(1)} \bullet m)_{(0)} \\
    & = (x_{(0)} \otimes w) \otimes x_{(1)} m_{(-1)} \otimes (x_{(2)} \bullet m_{(0)}) \\
    & = (x_{(0)} \otimes w_{(0)}) \otimes x_{(1)} w_{(1)} \otimes (x_{(2)} \bullet m) \\
    & = (x_{(0)} \otimes w_{(0)})_{(0)} \otimes (x_{(0)} \otimes w_{(0)})_{(1)} \otimes (x_{(2)} \bullet m),
  \end{align*}
  where the first and the third equality follow from \eqref{eq:C-H-Mod-D-axiom} and the definition of the coaction on $X \otimes W$, respectively. The above computation shows that the image of $X \otimes (W \coten_C M)$ under \eqref{eq:T-M-structure-morphism} is contained in $(X \otimes W) \coten_C M$. Namely, the linear map $\xi^{\bullet}_{X,W}$ is well-defined.

  It is straightforward to check that $\xi^{\bullet}_{X,W}$ is a morphism in $\Mod^D$ by using \eqref{eq:C-H-Mod-D-axiom}. It is also easy to see that $\xi^{\bullet}_{X,W}$ is natural in $X \in \Mod^H$ and $W \in \Mod^C$. The axiom~\eqref{eq:def-lax-module-functor} of lax module functors follows from the associativity and the unitality of the action $\bullet$. The proof is done.
\end{proof}

We prove Lemma~\ref{lem:equivariant-EW-2}. We fix $M \in {}^C\Mod^D$ and assume that $T_M : \Mod^C \to \Mod^D$ is a lax $\Mod^H$-module functor together with structure morphism $\xi$. Given $X \in \Vect$, we denote by $X_0$ the right $H$-comodule whose underlying vector space is $X$ and the coaction is given by $x \mapsto x \otimes 1_H$ ($x \in X$). For all $X \in \Vect$ and $W \in \Mod^C$, we have
\begin{equation*}
  (X_0 \otimes W) \coten_C M = X_0 \otimes (W \coten_C M)
\end{equation*}
as subspaces of $X \otimes W \otimes M$. Namely, the source and the target of the map $\xi_{X_0,W}$ are the same. We first remark the following easy fact:

\begin{claim}
  For all $X \in \Vect$ and $W \in \Mod^C$, we have
  \begin{equation}
    \label{eq:equivariant-EW-1-xi-W-x0}
    \xi_{X_0, W} = \id_{X_0 \otimes (W \coten_C M)}.
  \end{equation}
\end{claim}
\begin{proof}
  This follows from \eqref{eq:def-lax-module-functor} and the fact that any vector space is a direct sum of the unit object of $\Vect$.
\end{proof}

Next, before showing that the operation $\bullet_{\xi}$ makes $M$ a left $H$-module, we prove that $\xi_{X,W}$ is induced by the linear map~\eqref{eq:T-M-structure-morphism} with $\bullet = \bullet_{\xi}$. For this purpose, we verify the following equation:

\begin{claim}
  \label{claim:equivariant-EW-2}
  For $m \in M$ and $h \in H$, we have
  \begin{equation}
    \label{eq:equivariant-EW-1-proof-1}
    (\id_H \otimes \varepsilon_C \otimes \id_M) \xi_{H,C} (h \otimes m_{(-1)} \otimes m_{(0)})
    = h_{(1)} \otimes (h_{(2)} \bullet_{\xi} m).
  \end{equation}    
\end{claim}
\begin{proof}
  Noting that the comultiplication $\Delta_H : H \to H_0 \otimes H$ is a morphism of right $H$-comodules, we compute the left-hand side of the equation as follows:
  \begin{align*}
    & (\id_H \otimes \varepsilon_C \otimes \id_M) \xi_{H,C} (h \otimes m_{(-1)} \otimes m_{(0)}) \\
    & = (\id_{H_0} \otimes \varepsilon_H \otimes \varepsilon_C \otimes \id_M)
      (\Delta_H \otimes \id_C \otimes \id_M) \xi_{H,C} (h \otimes m_{(-1)} \otimes m_{(0)}) \\
    & = (\id_{H_0} \otimes \varepsilon_H \otimes \varepsilon_C \otimes \id_M)
      \xi_{H_0 \otimes H,C} (\Delta_H(h) \otimes m_{(-1)} \otimes m_{(0)}) \\
    & = (\id_{H_0} \otimes \varepsilon_H \otimes \varepsilon_C \otimes \id_M)
      (\id_{H_0} \otimes \xi_{H,C}) (h_{(1)} \otimes h_{(2)} \otimes m_{(-1)} \otimes m_{(0)}) \\
    & = h_{(1)} \otimes (h_{(2)} \bullet_{\xi} m).
  \end{align*}
  Here, the first equality follows from the counitality of $\Delta_H$,
  the second from the naturality of $\xi$,
  the third from \eqref{eq:def-lax-module-functor} and \eqref{eq:equivariant-EW-1-xi-W-x0}, and the last from the definition \eqref{eq:equivariant-EW-1-action-def} of the operation $\bullet_{\xi}$.
\end{proof}

As stated prior to Claim~\ref{claim:equivariant-EW-2}, we now prove:

\begin{claim}
  \label{claim:equivariant-EW-3}
  For $X \in \Mod^H$ and $W \in \Mod^C$, the map $\xi_{X,W}$ is induced by
  \begin{equation}
    \label{eq:equivariant-EW-1}
    X \otimes W \otimes M
    \to X \otimes W \otimes M, \quad
    x \otimes w \otimes m
    \mapsto x_{(0)} \otimes w \otimes (x_{(1)} \bullet_{\xi} m).
  \end{equation}
\end{claim}
\begin{proof}
  We first verify this for $(X, W) = (H, C)$.
  By~\eqref{eq:def-lax-module-functor} and \eqref{eq:equivariant-EW-1-xi-W-x0}, we have
  \begin{equation}
    \label{eq:equivariant-EW-3-proof-1}
    \xi_{H, C_0 \otimes C} = \xi_{H, C_0 \otimes C} \circ (\id_H \otimes \xi_{C_0, C})
    = \xi_{H \otimes C_0, C}.
  \end{equation}
  Given vector spaces $X$ and $Y$, we denote by $\tau_{X,Y} : X \otimes Y \to Y \otimes X$ the linear map defined by flipping the first and the second tensor component.
  Since $\tau_{H, C_0}$ is a morphism in $\Mod^H$, we have
  \begin{align*}
    \xi_{H, C_0 \otimes C}
    & = (\tau_{C_0,H} \tau_{H, C_0} \otimes \id_C \otimes \id_M) \circ \xi_{H \otimes C_0, C} \\
    & = (\tau_{C_0,H} \otimes \id_C \otimes \id_M) \circ \xi_{C_0 \otimes H, C}
      \circ (\tau_{H,C_0} \otimes \id_C \otimes \id_M) \\
    & = (\tau_{C_0,H} \otimes \id_C \otimes \id_M) \circ (\id_{C_0} \otimes \xi_{H, C})
      \circ (\tau_{H,C_0} \otimes \id_C \otimes \id_M),
  \end{align*}
  where the first equality follows from \eqref{eq:equivariant-EW-3-proof-1}, the second from the naturality of $\xi$, and the last from \eqref{eq:def-lax-module-functor} and \eqref{eq:equivariant-EW-1-xi-W-x0}. Hence, by \eqref{eq:equivariant-EW-1-proof-1}, we have
  \begin{align*}
    & (\id_{H} \otimes \id_{C_0} \otimes \varepsilon_C \otimes \id_M)
      \xi_{H, C_0 \otimes C}(h \otimes c \otimes m_{(-1)} \otimes m_{(0)}) \\
    & = (\tau_{C_0,H} \otimes \varepsilon_C \otimes \id_M)
      (\id_{C_0} \otimes \xi_{H, C})
      (c \otimes h \otimes m_{(-1)} \otimes m_{(0)}) \\
    & = h_{(1)} \otimes c \otimes (h_{(2)} \bullet_{\xi} m)
  \end{align*}
  for $h \in H$, $c \in C_0$ ($=C$) and $m \in M$.
  Noting that $\Delta_C : C \to C_0 \otimes C$ is a morphism in $\Mod^C$, we verify the claim for $(X, W) = (H, C)$ as follows:
  \begin{align*}
    & \xi_{H, C}(h \otimes m_{(-1)} \otimes m_{(0)}) \\
    & = (\id_H \otimes \id_{C_0} \otimes \varepsilon_C \otimes \id_M)
      (\id_H \otimes \Delta_C \otimes \id_M)
      \xi_{H, C}(h \otimes m_{(-1)} \otimes m_{(0)}) \\
    & = (\id_H \otimes \id_{C_0} \otimes \varepsilon_C \otimes \id_M)
      \xi_{H, C_0 \otimes C}(h \otimes m_{(-2)} \otimes m_{(-1)} \otimes m_{(0)}) \\
    & = h_{(1)} \otimes m_{(-1)} \otimes (h_{(2)} \bullet_{\xi} m_{(0)}).
  \end{align*}

  Now we consider the general case.
  By the naturality of $\xi$, the claim is also true if $W$ and $X$ are free comodules over $C$ and $H$, respectively. Since every comodule can be embedded into a free comodule, the claim is true for all objects $W \in \Mod^C$ and $X \in \Mod^H$.
\end{proof}

\begin{proof}[Proof of Lemma~\ref{lem:equivariant-EW-2}]
  Although the computation is a bit technical, the proof is done just by translating the assumption that $T_M$ is a lax module functor with structure morphism $\xi$ in terms of the operation $\bullet_{\xi}$.
  We first verify that $\bullet_{\xi}$ is associative.
  For $h, h' \in H$ and $m \in M$, we have
  \begin{align*}
    & \xi_{H \otimes H, C}(h \otimes h' \otimes m_{(-1)} \otimes m_{(0)}) \\
    & = (h \otimes h')_{(0)} \otimes m_{(-1)} \otimes ((h \otimes h')_{(1)} \bullet_{\xi} m_{(0)}) \\
    & = h_{(1)} \otimes h'_{(1)} \otimes m_{(-1)} \otimes ((h_{(2)} h'_{(2)}) \bullet_{\xi} m_{(0)})
  \end{align*}
  by Claim~\ref{claim:equivariant-EW-3}. By \eqref{eq:def-lax-module-functor}, we also compute
  \begin{align*}
    & \xi_{H \otimes H, C}(h \otimes h' \otimes m_{(-1)} \otimes m_{(0)}) \\
    & = \xi_{H, H \otimes C} (\id_H \otimes \xi_{H, C}) (h \otimes h' \otimes m_{(-1)} \otimes m_{(0)}) \\
    & = \xi_{H, H \otimes C} (h \otimes h'_{(1)} \otimes m_{(-1)} \otimes (h'_{(2)} \bullet_{\xi} m_{(0)})) \\
    & = h_{(1)} \otimes h'_{(1)} \otimes m_{(-1)} \otimes (h_{(2)} \bullet_{\xi} (h'_{(2)} \bullet_{\xi} m_{(0)})).
  \end{align*}
  Thus we have obtained
  \begin{gather*}
    h_{(1)} \otimes h'_{(1)} \otimes m_{(-1)} \otimes ((h_{(2)} h'_{(2)}) \bullet_{\xi} m_{(0)}) \\
    = h_{(1)} \otimes h'_{(1)} \otimes m_{(-1)} \otimes (h_{(2)} \bullet_{\xi} (h'_{(2)} \bullet_{\xi} m_{(0)})).
  \end{gather*}
  By applying $\varepsilon_H \otimes \varepsilon_H \otimes \varepsilon_C \otimes \id_M$ to the both sides, we obtain
  \begin{equation*}
    (h h') \bullet_{\xi} m = h \bullet_{\xi} (h' \bullet_{\xi} m),
  \end{equation*}
  that is, $\bullet_{\xi}$ is associative.

  To show that the action $\bullet_{\xi}$ is unital, we consider the linear map $u: \bfk \to H$ defined by $u(1_{\bfk}) = 1_H$. If $\bfk$ is viewed as a right $H$-comodule by the coaction $1_{\bfk} \mapsto 1_{\bfk} \otimes 1_H$, then $u$ is a morphism in $\Mod^H$. We fix an element $m \in M$. By Claim \ref{claim:equivariant-EW-3} and the axiom \eqref{eq:def-lax-module-functor} of lax module functors, we have
  \begin{gather*}
    1_H \otimes m_{(-1)} \otimes (1_H \bullet_{\xi} m_{(0)})
    = \xi_{H,C}(1_H \otimes m_{(-1)} \otimes m_{(0)}) \\
    = (u \otimes \id_C \otimes \id_M) \xi_{\bfk, C}(1_{\bfk} \otimes m_{(-1)} \otimes m_{(0)})
    = 1_H \otimes m_{(-1)} \otimes m_{(0)},
  \end{gather*}
  which implies $1_{H} \bullet_{\xi} m = m$.

  Finally, we check the compatibility condition \eqref{eq:C-H-Mod-D-axiom} between the action $\bullet_{\xi}$ and the coactions on $M$. We fix elements $h \in H$ and $m \in M$.
  Since the map $\xi_{H,C}$ is a morphism in $\Mod^D$, we have
  \begin{align*}
    & h_{(1)} \otimes m_{(-1)} \otimes \coactr_{M} (h_{(2)} \bullet_{\xi} m_{(0)}) \\
    & = \coactr_{(H \otimes W) \coten_C M} (h_{(1)} \otimes m_{(-1)} \otimes (h_{(2)} \bullet_{\xi} m_{(0)})) \\
    & = \coactr_{(H \otimes W) \coten_C M} \xi_{H,C}(h \otimes m_{(-1)} \otimes m_{(0)}) \\
    & = (\xi_{H,C} \otimes \id_{D}) \coactr_{H \otimes (W \coten_C M)}(h \otimes m_{(-1)} \otimes m_{(0)}) \\
    & = \xi_{H,C}(h_{(1)} \otimes m_{(-1)} \otimes m_{(0)}) \otimes h_{(2)} m_{(1)},
  \end{align*}
  where the first and the last equality follow from the definition of the coaction on $(H \otimes W) \coten_C M$ and $H \otimes (W \coten_C M)$, respectively.
  By applying $\varepsilon_H \otimes \varepsilon_C \otimes \id_{M} \otimes \id_D$ to the both sides, we obtain
  \begin{equation*}
    \coactr_{M}(h \bullet_{\xi} m)
    = (h_{(1)} \bullet_{\xi} m_{(0)}) \otimes h_{(2)} m_{(1)}.
  \end{equation*}
  Since the image of $\xi_{H,C}$ is contained in $(H \otimes C) \coten_C M$, we also have
  \begin{align*}
    & h_{(1)} \otimes m_{(-1)} \otimes \coactl_M(h_{(2)} \bullet_{\xi} m) \\
    & = (\id_H \otimes \id_C \otimes \coactl_M) \xi_{H,C}(h \otimes m_{(-1)} \otimes m_{(0)}) \\
    & = (\coactr_{H \otimes C} \otimes \id_M) \xi_{H,C}(h \otimes m_{(-1)} \otimes m_{(0)}) \\
    & = h_{(1)} \otimes m_{(-2)} \otimes h_{(2)} m_{(-1)} \otimes (h_{(3)} \bullet_{\xi} m_{(0)}).
  \end{align*}
  By applying $\varepsilon_H \otimes \varepsilon_C \otimes \id_C \otimes \id_M$ to the both sides, we obtain
  \begin{equation*}
    \coactl_{M}(h \bullet_{\xi} m)
    = h_{(1)} m_{(-1)} \otimes (h_{(2)} \bullet_{\xi} m_{(1)}). \qedhere
  \end{equation*}
\end{proof}

\begin{proof}[Proof of Theorem~\ref{thm:equivariant-EW}]
  Let $M$ be an object of ${}_H^C\Mod^D$ with left $H$-action $\bullet$, and write $\zeta = \xi^{\bullet}$. Then the action $\bullet_{\zeta}$ is given by  
  \begin{align*}
    h \bullet_{\zeta} m
    & = (\varepsilon_H \otimes \varepsilon_C \otimes \id_M)
      \zeta_{H, C}(h \otimes m_{(-1)} \otimes m_{(0)}) \\
    & = (\varepsilon_H \otimes \varepsilon_C \otimes \id_M)
      (h_{(1)} \otimes m_{(-1)} \otimes h_{(2)} \bullet m_{(0)})
      = h \bullet m
  \end{align*}
  for $h \in H$ and $m \in M$.
  Thus, together with Claim~\ref{claim:equivariant-EW-3}, we conclude that the map $\bullet \mapsto \xi^{\bullet}$ and $\xi \mapsto \bullet_{\xi}$ give bijections between the sets (1) and (2) of Theorem~\ref{thm:equivariant-EW}.

  Now let $M$ and $N$ be objects of ${}^C_H\Mod^D$ and let $f: M \to N$ be a morphism of $C$-$D$-bicomodules. One can verify that $f$ is a morphism in ${}^C_H\Mod^D$ if and only if the natural transformation $T_M \to T_N$ induced by $f$ is a morphism of lax $\Mod^H$-module functors. The proof is completed.
\end{proof}

\subsection{Variants of Theorem~\ref{thm:equivariant-EW}}
\label{subsec:equivariant-EW-variants}
Let $K$ also be a bialgebra, and let $C$ and $D$ be right $K$-module coalgebras.
Then the categories $\Mod^C$ and $\Mod^D$ are right module categories over $\Mod^K$.
If $M \in {}^C\Mod^D_K$ ($:=$ the category of $C$-$D$-bicomodules in the monoidal category $\Mod_K$ of right $K$-modules), then the functor $T_M$ has a structure of a right $\Mod^K$-module functor induced by the linear map
\begin{equation*}
  W \otimes M \otimes X \to W \otimes X \otimes M,
  \quad w \otimes m \otimes x \mapsto
  w \otimes x_{(0)} \otimes m x_{(1)},
\end{equation*}
where $w \in W \in \Mod^C$, $x \in X \in \Mod^K$ and $m \in M$.
By applying Theorem~\ref{thm:equivariant-EW} to $H = K^{\op}$, we see that the assignment $M \mapsto T_M$ gives an equivalence between ${}^C\Mod^D_K$ the category of lax right $\Mod^K$-module functors from $\Mod^C$ to $\Mod^D$ whose underlying functor is of type $\mathscr{L}$.

Given bialgebras $H$ and $K$, the category ${}_H\Mod_K$ of $H$-$K$-bimodules is a monoidal category with the monoidal product inherited from ${}_H\Mod$ and $\Mod_K$.
Theorem~\ref{thm:equivariant-EW} has the following bimodule version:
If $C$ and $D$ are coalgebras in ${}_H\Mod_K$, then $\Mod^C$ and $\Mod^D$ are $\Mod^H$-$\Mod^K$-bimodule categories. The assignment $M \mapsto T_M$ gives an equivalence between the category ${}^C_H\Mod^D_K$ of $C$-$D$-bicomodules in ${}_H\Mod_K$ and the category of lax $\Mod^H$-$\Mod^K$-bimodule functors from $\Mod^C$ to $\Mod^D$ whose underlying functor is of type $\mathscr{L}$.

\section{Applications to Hopf and Yetter-Drinfeld modules}
\label{sec:applications-Hopf-YD}

\subsection{When is a lax module functor strong?}

In this section, we explain how known results on Hopf modules and Yetter-Drinfeld modules are deduced from our equivariant Eilenberg-Watts theorem. Let $H$ be a bialgebra, and let $C$ and $D$ be left $H$-module coalgebras. A central role is played by the inclusion functor
\begin{equation*}
  i_{C, D}:
  \mathscr{L}_{\Mod^H}^{\strong}(\Mod^C, \Mod^D)
  \to \mathscr{L}_{\Mod^H}^{\lax}(\Mod^C, \Mod^D),
\end{equation*}
where the source is the full subcategory of strong module functors. The existence of an antipode of $H$ is closely related to when the inclusion functor $i_{C,D}$ is an equivalence, or when the equation `lax = strong' holds. More precisely, we have:

\begin{theorem}
  \label{thm:equivalence-iota}
  For a bialgebra $H$, the following are equivalent:
  \begin{enumerate}
  \item $H$ is a Hopf algebra, that is, it has an antipode.
  \item $i_{C,D}$ is an equivalence for any left $H$-module coalgebras $C$ and $D$.
  \item There exist left $H$-module coalgebras $C$ and $D$ such that $i_{C,D}$ is an equivalence, $C$ is non-zero, and the vector space $C \otimes H$ is an object of ${}^C_H\Mod^D$ together with left $C$-coaction $\Delta_C \otimes \id_H$ and the diagonal left $H$-action.
  \end{enumerate}
\end{theorem}

We add a remark on Condition (3). Let $C$ and $D$ be non-zero left $H$-module coalgebras. Then the vector space $C \otimes H$ becomes an object of ${}^C_H\Mod^D$ in the manner specified in Condition (3) if there is a homomorphism $\pi : H \to D$ of left $H$-module coalgebras. Indeed, equip it with the right $D$-coaction $\id_C \otimes (\id_H \otimes \pi)\Delta_H$. For example, the pairs $(C, D) = (H, H)$ and $(C, D) = (H, \bfk)$ satisfy Condition (3) except the part that $i_{C,D}$ is an equivalence.

\begin{proof}
  We first show that (1) implies (2). We assume that $H$ has an antipode. 
  It suffices to show that every object $F \in \mathscr{L}_{\Mod^H}^{\lax}(\Mod^C, \Mod^D)$ is a strong module functor.
  By Theorem~\ref{thm:equivariant-EW}, we may assume that $F = \Phi_{C,D}(M)$ for some $M \in {}^C_H\Mod^D$ and the structure morphism $\xi$ of $F$ is induced by \eqref{eq:T-M-structure-morphism}. Now we define
  \begin{equation*}
    \overline{\xi}_{X,W} : (X \otimes W) \coten_C M \to X \otimes (W \coten_C M)
    \quad (X \in \Mod^H, W \in \Mod^C)
  \end{equation*}
  to be the linear map induced by the linear map
  \begin{align*}
      (X \otimes W) \otimes M
      & \to X \otimes (W \otimes M), \\
      (x \otimes w) \otimes m
      & \mapsto x_{(0)} \otimes (w \otimes S(x_{(1)}) m)
      \quad (x \in X, w \in W, m \in M),
  \end{align*}
  where $S$ is the antipode of $H$. The well-definedness of $\overline{\xi}_{X,W}$ is explained as follows: Let $t$ be an element of the source of $\overline{\xi}_{X,M}$. Although $t$ may not be a simple tensor, we write it as $t = x \otimes w \otimes m$ for brevity. As it belongs to the cotensor product, we have $x_{(0)} \otimes w_{(0)} \otimes x_{(1)} w_{(1)} \otimes m = x \otimes w \otimes m_{(-1)} \otimes m_{(0)}$. By applying the coaction of $C$ on $X$ to the both sides, we obtain
  \begin{equation*}
    x_{(0)} \otimes x_{(1)} \otimes w_{(0)} \otimes x_{(2)} w_{(1)} \otimes m = x_{(0)} \otimes x_{(1)} \otimes w \otimes m_{(-1)} \otimes m_{(0)}.
  \end{equation*}
  Now we use this to verify $\overline{\xi}_{X,W}(t) \in X \otimes (W \coten_C M)$ as follows:
  \begin{align*}
    & x_{(0)} \otimes (w \otimes (S(x_{(1)}) m)_{(-1)} \otimes (S(x_{(1)}) m)_{(0)}) \\
    & = x_{(0)} \otimes (w \otimes S(x_{(1)})_{(1)} m_{(-1)} \otimes S(x_{(1)})_{(2)} m_{(0)}) \\
    & = x_{(0)} \otimes (w_{(0)} \otimes S(x_{(1)})_{(1)} x_{(2)} w_{(1)} \otimes S(x_{(1)})_{(2)} m) \\
    & = x_{(0)} \otimes (w_{(0)} \otimes S(x_{(2)}) x_{(3)} w_{(1)} \otimes S(x_{(1)}) m) \\
    & = x_{(0)} \otimes (w_{(0)} \otimes w_{(-1)} \otimes S(x_{(1)}) m).
  \end{align*}
  Once we know that $\overline{\xi}_{X,W}$ is a well-defined linear map, it is straightforward to check that $\xi_{X,W}$ and $\overline{\xi}_{X,W}$ are mutually inverse. We have proved (1) $\Rightarrow$ (2).

  The implication (2) $\Rightarrow$ (3) is obvious. It remains to show (3) $\Rightarrow$ (1). We assume that (3) holds and let $C$ and $D$ be as in the assertion (3). We consider the lax $\Mod^H$-module functor $T_M : \Mod^C \to \Mod^D$ induced by the object $M := C \otimes H \in {}_H^C\Mod^D$. The structure morphism of $T_M$, which we denote by
  \begin{equation*}
    \xi_{X,W} : X \otimes T_M(W) \to T_M(X \otimes W) \quad (X \in \Mod^H, W \in \Mod^C),
  \end{equation*}
  is induced by the linear map $X \otimes W \otimes C \otimes H \to X \otimes W \otimes C \otimes H$ given by
  \begin{gather*}
    x \otimes w \otimes c \otimes h \mapsto x_{(0)} \otimes w \otimes x_{(1)} c \otimes x_{(2)} h
  \end{gather*}
  for $x \in X$, $w \in W$, $c \in C$ and $h \in H$. Since $M$ is free as a left $C$-comodule, we have the following natural isomorphism of vector spaces:
  \begin{equation*}
    \phi_{W} : W \otimes H \to T_M(W) = W \coten_C M,
    \quad w \otimes h \mapsto w_{(0)} \otimes w_{(1)} \otimes h
    \quad (W \in \Mod^C).
  \end{equation*}
  Since the functor $i_{C,D}$ is an equivalence by our assumption, the map $\xi_{X,W}$ is invertible for all $X \in \Mod^H$ and $W \in \Mod^C$. Hence the map
  \begin{equation*}
    \beta_{X,W} := \phi_{X \otimes W}^{-1} \circ \xi_{X,W} \circ (\id_X \otimes \phi_W) : X \otimes W \otimes H \to X \otimes W \otimes H
  \end{equation*}
  is also invertible. One easily has
  \begin{equation*}
    \beta_{H,C}(h \otimes c \otimes h')
    = h_{(1)} \otimes c \otimes h_{(2)} h'
    \quad (c \in C, h, h' \in H).
  \end{equation*}
  Since $C \ne 0$ by our assumption, the linear map
  \begin{equation*}
    \beta : H \otimes H \to H \otimes H, \quad h \otimes h' \mapsto h_{(1)} \otimes h_{(2)} h'
    \quad (h, h' \in H)
  \end{equation*}
  is invertible. A well-known argument (see, {\it e.g.} \cite[Proposition 1.2.19]{MR4164719}) shows that the linear map $h \mapsto (\varepsilon \otimes \id_H) \beta^{-1}(h \otimes 1_H)$ ($h \in H$) is an antipode of $H$. The proof is done.
\end{proof}

\subsection{The fundamental theorem for Hopf modules}

Let $H$ be a bialgebra. We discuss {\em the fundamental theorem for Hopf modules} \cite[Theorem 4.4.6]{MR1786197}, which states that there is a category equivalence $\Vect \approx {}_H^H\Mod$ if $H$ is a Hopf algebra.
For a left $H$-comodule algebra $D$, we introduce the functor
\begin{equation*}
  \Psi_D : \Mod^D \to {}^H_H \Mod^D, \quad \Psi_D(W) = H \otimes W \quad (W \in \Mod^D),
\end{equation*}
where the left coaction of $H$, the right coaction of $D$ and the left action of $H$ on the vector space $\Psi_D(W)$ are given by
\begin{equation*}
  \coactl_{\Psi_D(W)} = \Delta_H \otimes \id_W, \quad
  \coactr_{\Psi_D(W)}(h \otimes w) = h_{(1)} \otimes w_{(0)} \otimes h_{(2)} w_{(1)},
\end{equation*}
and $h \cdot (h' \otimes w) = h_{(1)} h' \otimes h_{(2)} w$, respectively, for $h, h' \in H$ and $w \in W$. Now we state the fundamental theorem for Hopf modules and related results as follows:

\begin{theorem}
  \label{thm:fundamental-theorem}
  For a bialgebra $H$, the following are equivalent:
  \begin{enumerate}
  \item $H$ is a Hopf algebra.
  \item The functor $\Psi_D$ is an equivalence for all left $H$-module coalgebras $D$.
  \item The functor $\Psi_{\bfk} : \Vect \to {}^H_H\Mod$ is an equivalence.
  \item The functor $\Psi_H : \Mod^H \to {}^H_H\Mod^H$ is an equivalence.
  \end{enumerate}
  If these equivalent conditions are satisfied, then we have
  \begin{equation}
    \label{eq:fundamental-theorem-q-inverse}
    \Psi_D^{-1}(M) = \{ m \in M \mid \coactl_M(m) = 1 \otimes m \}
    \quad (M \in {}^H_H\Mod^D),
  \end{equation}
  which is a right $D$-comodule as a $D$-subcomodule of $M$.
\end{theorem}

The implication (1) $\Rightarrow$ (3) is just the fundamental theorem of Hopf modules \cite[Theorem 4.4.6]{MR1786197}. While the authors were unable to locate a precise reference, it is well-known that the converse of the fundamental theorem holds true; see, {\it e.g.}, \cite[Introduction]{MR2037710}. The equivalence between (1) and (4) has been verified in a more general setting of quasi-bialgebras in \cite{MR2037710} and \cite{MR3004083}.
Thus, the most part of Theorem~\ref{thm:fundamental-theorem} is not new.
What we would like to emphasize here is that Theorem~\ref{thm:fundamental-theorem} is a consequence of Theorem \ref{thm:equivalence-iota}, which describes when the equation `lax = strong' holds.

\begin{proof}[Proof of Theorem~\ref{thm:fundamental-theorem}]
  It is easy to see that the functor
  \begin{equation*}
    t_D : \Mod^D \to \mathscr{L}_{\Mod^H}^{\strong}(\Mod^H, \Mod^D),
    \quad t_D(W) = (-) \otimes W
  \end{equation*}
  is an equivalence with a quasi-inverse given by $t_D^{-1}(F) = F(\bfk)$, where $\bfk$ is the trivial right $H$-comodule. We note that the composition
  \begin{equation}
    \label{eq:fundamental-theorem-functor}
    \Mod^D
    \xrightarrow[\approx]{\  t_D \ }
    \mathscr{L}_{\Mod^H}^{\strong}(\Mod^H, \Mod^D)
    \xrightarrow{\ i_{H,D} \ }
    \mathscr{L}_{\Mod^H}^{\lax}(\Mod^H, \Mod^D)
    \xrightarrow[\approx]{\ \Phi_{H,D}^{-1} \ }
    {}^H_H\Mod^D
  \end{equation}
  is isomorphic to the functor $\Psi_D$. Thus $\Psi_D$ is an equivalence if and only if $i_{H,D}$ is an equivalence. Now we see that (1), (2), (3) and (4) are equivalent by rephrasing Theorem~\ref{thm:equivalence-iota} with noting that $\id_H : H \to H$ and $\varepsilon_H : H \to \bfk$ are morphisms of left $H$-module coalgebras.

  Now we suppose that (1), (2), (3) and (4) hold. Since $\Psi_D$ is isomorphic to \eqref{eq:fundamental-theorem-functor}, a quasi-inverse of $\Psi_D$ is given by the composition
  \begin{equation*}
    {}^H_H\Mod^D
    \xrightarrow[\approx]{\ \Phi_{H,D} \ }
    \mathscr{L}_{\Mod^H}^{\lax}(\Mod^H, \Mod^D)
    \xrightarrow{\ i_{H,D}^{-1} \ }
    \mathscr{L}_{\Mod^H}^{\strong}(\Mod^H, \Mod^D)
    \xrightarrow[\approx]{\ t_D^{-1} \ }
    \Mod^D,
  \end{equation*}
  which sends $M \in {}^H_H\Mod^D$ to $\bfk \coten_H M$.
  Since $\bfk \coten_H M$ is identified with the space of left $H$-coinvariants of $M$, the formula~\eqref{eq:fundamental-theorem-q-inverse} of $\Psi_D^{-1}$ follows. The proof is done.
\end{proof}

\subsection{Equivalence for Yetter-Drinfeld categories}

\newcommand{\YD}{\mathscr{Y}\!\mathscr{D}}

We explain how known equivalences for Yetter-Drinfeld categories are naturally obtained from the equation `lax = strong' mentioned in the above.
Let $H$, $A$ and $C$ be a bialgebra, an $H$-bicomodule algebra, and an $H$-bimodule coalgebra, respectively. Then the category $\YD^C_A$ of Yetter-Drinfeld modules is defined as follows: An object of this category is a right $A$-module $M$ equipped with a right $C$-comodule structure such that the equation
\begin{equation}
  \label{eq:YD-condition}
  (m a_{(0)})_{(0)} \otimes a_{(-1)} \triangleright (m a_{(0)})_{(1)}
  = m_{(0)} a_{(0)} \otimes m_{(1)} \triangleleft a_{(1)}
  \quad (a \in A, m \in M)
\end{equation}
holds, where $\triangleright$ and $\triangleleft$ are the left and the right action of $H$ on $C$. A morphism in $\YD^C_A$ is a linear map preserving the action of $A$ and the coaction of $C$. If $H$ has an antipode, then the condition~\eqref{eq:YD-condition} is equivalent to
\begin{equation*}
  \coactr_M(m a)
  = m_{(0)} a_{(0)} \otimes S(a_{(-1)}) \triangleright m_{(1)} \triangleleft a_{(1)}
  \quad (a \in A, m \in M).
\end{equation*}

The category $\YD^C_A$ for $A = H$ has the following interpretation: Given a bimodule category $\mathcal{M}$ over a monoidal category $\mathcal{C}$, the lax center $\mathscr{Z}_{\mathcal{C}}^{\lax}(\mathcal{M})$ is defined to be the category whose object is a pair $(M, \sigma)$ consisting of an object $M \in \mathcal{M}$ and a natural transformation $\sigma_X : M \otimes X \to X \otimes M$ ($X \in \mathcal{C}$) satisfying
\begin{equation*}
  \sigma_{X \otimes Y} = (\id_X \otimes \sigma_Y) (\sigma_X \otimes \id_Y)
  \quad \text{and} \quad \sigma_{\unitobj} = \id_M
\end{equation*}
for all objects $X, Y \in \mathcal{M}$. A morphism $f: (M, \sigma) \to (N, \tau)$ in $\mathscr{Z}_{\mathcal{C}}^{\lax}(\mathcal{M})$ is a morphism $f: M \to N$ in $\mathcal{M}$ satisfying $(\id_X \otimes f) \circ \sigma_X = \tau_X \circ (f \otimes \id_X)$ for all $X \in \mathcal{C}$. An object $M \in \YD_H^C$ becomes an object of the category $\mathscr{Z}_{\Mod^H}^{\lax}(\Mod^C)$ together with the natural transformation defined by
\begin{equation*}
  \sigma_{M, X} : M \otimes X \to X \otimes M,
  \quad m \otimes x \mapsto x_{(0)} \otimes m x_{(1)}
  \quad (m \in M, x \in X \in \Mod^H)
\end{equation*}
and, as in the well-known case where $C = H$, we see that the map $M \mapsto (M, \sigma_{M,-})$ gives rise to an equivalence from $\YD_H^C$ to $\mathscr{Z}_{\Mod^H}^{\lax}(\Mod^C)$. Now, based on this equivalence, we prove:

\begin{theorem}
  \label{thm:YD-equivalence}
  If $H$ is a Hopf algebra, then $\YD^C_H \approx {}^H_H\Mod^C_H$.
\end{theorem}

This a special case of a well-known equivalence $\YD^C_A \approx {}^H_H\Mod^C_A$ established in \cite[Theorem 2.3]{MR1469112}. Similarly to Theorem~\ref{thm:fundamental-theorem}, the proof below aims to demonstrate that the equivalence $\YD^C_H \approx {}^H_H\Mod^C_H$ is a consequence of the equation `lax = strong' due to Theorem~\ref{thm:equivalence-iota}.

Recently, Aziz and Vercruysse \cite{AV} gave a module-category-theoretic interpretation of $\YD^C_A$ not limited to the case of $A = H$. Their result \cite[Theorem 4.3]{AV} involves a category with right `oplax action' of $\Mod^H$, that is, a category $\mathcal{M}$ equipped with an oplax monoidal functor from $(\Mod^H)^{\rev}$ to the category of endofunctors on $\mathcal{M}$. It is an interesting question whether \cite[Theorem 4.3]{AV} is obtained in a similar manner as the proof below with extending Theorems~\ref{thm:equivariant-EW} and~\ref{thm:equivalence-iota} to categories with (op)lax action of $\Mod^H$ and module functors between them.

\begin{proof}[Proof of Theorem~\ref{thm:YD-equivalence}]
  Let $\mathcal{L}$ be the category of $\Mod^H$-bimodule functors from $\Mod^H$ to $\Mod^C$ whose underlying functor is of type $\mathscr{L}$. As we have noted in Subsection~\ref{subsec:equivariant-EW-variants}, the category $\mathcal{L}$ is equivalent to ${}^H_H\Mod^C_H$. Thus, to prove this theorem, it suffices to show that $\mathcal{L}$ is equivalent to the lax center $\mathcal{Z} := \mathscr{Z}_{\Mod^H}^{\lax}(\Mod^C)$.

  By definition, an object of the category $\mathcal{L}$ is a functor $F : \Mod^H \to \Mod^C$ of type $\mathscr{L}$ endowed with natural transformations
  \begin{equation*}
    \xi^{\ell}_{X,Y} : X \otimes F(Y) \to F(X \otimes Y)
    \quad \text{and} \quad
    \xi^{r}_{X,Y} : F(X) \otimes Y \to F(X \otimes Y)
  \end{equation*}
  for $X, Y \in \Mod^H$ subject to the axioms for module functors.
  Theorem \ref{thm:equivalence-iota} implies that $\xi^{\ell}$ is invertible. Thus the right $C$-comodule $M := F(\bfk)$ comes equipped with a natural transformation $\sigma_{X} = (\xi_{X, \bfk}^{\ell})^{-1} \circ \xi_{\bfk, X}^r$ ($X \in \Mod^H$).
  By the axioms for module functors, it is easy to verify that $(M, \sigma)$ is an object of $\mathcal{Z}$. Hence we have a functor from $\mathcal{L}$ to $\mathcal{Z}$.

  Given an object $(M, \sigma)$ of $\mathcal{Z}$, we consider the functor $F : \Mod^H \to \Mod^C$ given by $F(X) = X \otimes M$ for $X \in \Mod^H$. The natural transformation $\sigma$ makes $F$ a lax bimodule functor of type $\mathscr{L}$. Thus we also have a functor from $\mathcal{Z}$ to $\mathcal{L}$. It is obvious that two functors between $\mathcal{L}$ and $\mathcal{Z}$ that we have constructed are mutually inverse to each other. The proof is done.
\end{proof}

\section{Bicategorical perspective}
\label{sec:applications-bicat}

\subsection{Bicategorical reformulation}

In this section, we interpret our results from a bicategorical perspective and give some applications. Let $H$ be a bialgebra. The equivalence of Theorem~\ref{thm:equivariant-EW} turns the cotensor product into the composition of lax module functors but with the order of the operations reversed. Namely, for left $H$-module coalgebras $C$, $D$ and $E$, the following diagram is commutative:
\begin{equation}
  \label{eq:Phi-and-composition}
  \begin{tikzcd}[column sep=5pt]
    {}_H^D\Mod^E \times {}_H^C\Mod^D
    \arrow[r, "\text{flip}"]
    \arrow[d, "{\Phi_{D,E} \times \Phi_{C,D}}"']
    & {}_H^C\Mod^D \times {}_H^D\Mod^E
    \arrow[r, "{\coten_D}"]
    & {}_H^C\Mod^E \arrow[d, "{\Phi_{C,E}}"] \\
    \mathscr{L}_{\Mod^H}^{\lax}(\Mod^D, \Mod^E)
    \times \mathscr{L}_{\Mod^H}^{\lax}(\Mod^C, \Mod^D)
    \arrow[rr, "\text{composition}"]
    & & \mathscr{L}_{\Mod^H}^{\lax}(\Mod^C, \Mod^E).
  \end{tikzcd}
\end{equation}
The observation is compiled into a biequivalence as follows: Given a bialgebra $H$, we introduce the following two bicategories:
\begin{itemize}
\item The bicategory $H\mbox{-}\mathbf{Coalg}$ has left $H$-module coalgebras as 0-cells. Given 0-cells $C$ and $D$ in this bicategory, ${}_H^C\Mod^D$ is the category of 1-cells from $C$ to $D$. The composition of 1-cells is given by the cotensor product.
\item The bicategory $\Mod^H\mbox{-}\mathbf{Mod}$ (which is in fact a 2-category) has left module categories over $\Mod^H$ that is equivalent to $\Mod^C$ for some left $H$-module coalgebra $C$ as 0-cells. Given 0-cells $\mathcal{M}$ and $\mathcal{N}$ of this bicategory, $\mathscr{L}^{\lax}_{\Mod^H}(\mathcal{M}, \mathcal{N})$ is the category of 1-cells from $\mathcal{M}$ to $\mathcal{N}$.
\end{itemize}
Given a bicategory $\mathcal{K}$, we denote by $\mathcal{K}^{\op}$ the bicategory obtained from $\mathcal{K}$ by reversing the order of the composition of 1-cells but leaving the order of the composition of 2-cells unchanged. The commutativity of the above diagram means:

\begin{theorem}
  \label{thm:biequiv-Coalg}
  There is a biequivalence
  \begin{equation*}
    H\mbox{-}\mathbf{Coalg} \to (\Mod^H\mbox{-}\mathbf{Mod})^{\op},
    \quad C \mapsto \Mod^C.
  \end{equation*}
\end{theorem}

\subsection{Equivalence of modules over $\Mod^H$ and ${}_H\Mod$}

Given a monoidal category $\mathcal{C}$, we denote by $\mathcal{C}^{\rev}$ the monoidal category obtained from $\mathcal{C}$ by reversing the order of the monoidal product. Theorem~\ref{thm:biequiv-Coalg} gives a monoidal equivalence
\begin{equation}
  \label{eq:H-Mod-as-dual}
  ({}_H\Mod, \otimes, \bfk)^{\rev}
  \approx (\mathscr{L}^{\lax}_{\Mod^H}(\Vect, \Vect), \circ, \id).
\end{equation}
For a left $H$-module coalgebra $C$, there are also equivalences
\begin{equation}
  \label{eq:L-Vect-Comod-C}
  {}_H\Mod^C
  \approx \mathscr{L}_{\Mod^H}^{\lax}(\Vect, \Mod^C), \quad
  {}_H^C\Mod
  \approx \mathscr{L}_{\Mod^H}^{\lax}(\Mod^C, \Vect)
\end{equation}
of categories. The category $\mathscr{L}_{\Mod^H}^{\lax}(\Vect, \Mod^C)$ is a left ${}_H\Mod$-module category by the action given by the monoidal equivalence \eqref{eq:H-Mod-as-dual} and the composition of lax module functors. The commutativity of \eqref{eq:Phi-and-composition} implies that the first equivalence in \eqref{eq:L-Vect-Comod-C} is in fact an equivalence of left ${}_H\Mod$-module categories. Similarly, the second equivalence in \eqref{eq:L-Vect-Comod-C} is an equivalence of right ${}_H\Mod$-module categories. Now we introduce the following 2-category:
\begin{itemize}
\item The 2-category ${}_H\Mod\mbox{-}\mathbf{Mod}$ has left ${}_H\Mod$-module categories that is equivalent to ${}_H\Mod^C$ for some left $H$-module coalgebra $C$ as 0-cells. Given 0-cells $\mathcal{M}$ and $\mathcal{N}$, the category of 1-cells from $\mathcal{M}$ to $\mathcal{N}$ is the category of left exact strong left $\Mod^H$-module functors from $\mathcal{M}$ to $\mathcal{N}$. 
\end{itemize}
The definition of ${}_H\Mod\mbox{-}\mathbf{Mod}$ is admittedly ad hoc as it is not compatible with the definition of $\Mod^H\mbox{-}\mathbf{Mod}$. Nevertheless, this 2-category is useful to state:

\begin{theorem}
  \label{thm:duality-1}
  There is a 2-equivalence
  \begin{equation}
    \label{eq:duality-1}
    (-)_{*} : \Mod^H\mbox{-}\mathbf{Mod}
    \to {}_H\Mod\mbox{-}\mathbf{Mod},
    \quad \mathcal{M} \mapsto \mathcal{M}_* := \mathscr{L}_{\Mod^H}^{\lax}(\Vect, \mathcal{M}),
  \end{equation}
  where ${}_H\Mod$ acts on the category $\mathcal{M}_{*}$ by the equivalence \eqref{eq:H-Mod-as-dual} and the composition of lax module functors.
\end{theorem}
\begin{proof}
  We shall first verify that \eqref{eq:duality-1} is a 2-functor. Let $F: \mathcal{M} \to \mathcal{N}$ be a 1-cell in the 2-category $\Mod^H\mbox{-}\mathbf{Mod}$. The only unclear point is that the induced functor
  \begin{equation*}
    F_{*} : \mathcal{M}_{*} \to \mathcal{N}_{*},
    \quad F_{*}(T) = F \circ T
    \quad (T \in \mathcal{M}_{*})
  \end{equation*}
  is indeed a 1-cell in ${}_H\Mod\mbox{-}\mathbf{Mod}$. It is obvious that $F_{*}$ is a strong ${}_H\Mod$-module functor. By the definition of $\Mod^H\mbox{-}\mathbf{Mod}$, we may assume that $\mathcal{M} = \Mod^C$ and $\mathcal{N} = \Mod^D$ for some left $H$-module coalgebras $C$ and $D$. We may also assume that $F = (-) \coten_C M$ for some $M \in {}_H^C\Mod^D$. Then $\mathcal{M}_{*}$, $\mathcal{N}_{*}$ and $F_{*}$ are identified with ${}_H\Mod^C$, ${}_H\Mod^D$ and $(-) \coten_C M$, respectively. Thus $F_{*}$ is left exact.

  We now show that \eqref{eq:duality-1} is a 2-equivalence. The 2-functor \eqref{eq:duality-1} is essentially surjective on 0-cells by \eqref{eq:L-Vect-Comod-C}. It remains to show that the functor
  \begin{equation}
    \label{eq:duality-1-proof-1}
    \mathscr{L}_{\Mod^H}^{\lax}(\mathcal{M}, \mathcal{N})
    \to
    \left(
      \begin{gathered}
        \text{the category of left exact strong} \\[-2pt]
        \text{left ${}_H\Mod$-module functors $\mathcal{M}_{*} \to \mathcal{N}_{*}$}
      \end{gathered}
    \right),
    \quad F \mapsto F_{*}
  \end{equation}
  is an equivalence for all 0-cells $\mathcal{M}$ and $\mathcal{N}$ in $\Mod^H\mbox{-}\mathbf{Mod}$. For this purpose, we may assume that $\mathcal{M} = \Mod^C$ and $\mathcal{N} = \Mod^D$ as above. Then the source of \eqref{eq:duality-1-proof-1} is identified with ${}_H^C\Mod^D$ by Theorem \ref{thm:equivariant-EW}, while the target of \eqref{eq:duality-1-proof-1} is also identified with ${}_H^C\Mod^D$ by Corollary~\ref{cor:EW-Pareigis}. Under these identifications, the functor~\eqref{eq:duality-1-proof-1} is identified with the identity functor, which is an equivalence. The proof is done.
\end{proof}

For a bialgebra $H$, we define the 2-category $\mathbf{Mod}\mbox{-}{}_H\Mod$ of right ${}_H\Mod$-module categories of the form ${}_H^C\Mod$ for some left $H$-module coalgebras $C$ by the same manner as ${}_H\Mod\mbox{-}\mathbf{Mod}$. By using the second equivalence in \eqref{eq:L-Vect-Comod-C} instead of the first one, one can prove the following theorem:

\begin{theorem}
  \label{thm:duality-2}
  There is a 2-equivalence
  \begin{equation}
    \label{eq:duality-2}
    (-)^{*} : \Mod^H\mbox{-}\mathbf{Mod}
    \to (\mathbf{Mod}\mbox{-}{}_H\Mod)^{\op},
    \quad \mathcal{M} \mapsto \mathcal{M}^* := \mathscr{L}_{\Mod^H}^{\lax}(\mathcal{M}, \Vect).
  \end{equation}
\end{theorem}

We assume that $H$ is a Hopf algebra. Then an inverse of the 2-equivalence \eqref{eq:duality-2} is given as follows: For a 0-cell $\mathcal{M}$ in $\mathbf{Mod}\mbox{-}{}_H\Mod$, we define $\mathcal{M}^{*}$ to be the category of left exact strong ${}_H\Mod$-module functors from $\Vect$ to $\mathcal{M}$, where ${}_H\Mod$ acts on $\Vect$ through the forgetful functor ${}_H\Mod \to \Vect$. We note that $\Vect \approx {}_H^H\Mod$ as right ${}_H\Mod$-module categories by the fundamental theorem for Hopf modules. Thus, if $\mathcal{M} = {}_H^C\Mod$ for some left $H$-module coalgebra $C$, then we have $\mathcal{M}^{*} \approx {}^H_H \Mod^C \approx \Mod^C$ by Corollary~\ref{cor:EW-Pareigis}. From this observation, we see that the 2-functor
\begin{equation*}
  \mathbf{Mod}\mbox{-}\Mod^H \to ({}_H\Mod\mbox{-}\mathbf{Mod})^{\op},
  \quad \mathcal{M} \mapsto \mathcal{M}^*
\end{equation*}
is an inverse of \eqref{eq:duality-2}.

We now assume moreover that the antipode of $H$ is bijective. Then an inverse of \eqref{eq:duality-1} is described in a similar way as above. The bijectivity of the antipode implies that $\Vect$ is equivalent to ${}_H\Mod^H$ as a left ${}_H\Mod$-module category. For a 0-cell $\mathcal{M}$ in ${}_H\Mod\mbox{-}\mathbf{Mod}$, we define $\mathcal{M}_{*}$ to be the category of left exact strong ${}_H\Mod$-module functors from $\Vect$ to $\mathcal{M}$. If $\mathcal{M} = {}_H\Mod^C$ for some left $H$-module coalgebra $C$, then we have $\mathcal{M}_{*} \approx {}^H_H \Mod^C \approx \Mod^C$ by Corollary~\ref{cor:EW-Pareigis}. Hence the 2-functor
\begin{equation*}
  {}_H\Mod\mbox{-}\mathbf{Mod} \to \Mod^H\mbox{-}\mathbf{Mod},
  \quad \mathcal{M} \mapsto \mathcal{M}_{*}
\end{equation*}
is an inverse of \eqref{eq:duality-1}.

\subsection{Equivariant Morita-Takeuchi equivalence}

Finally, we give applications to `equivariant' Morita-Takeuchi theory.
Let $H$ be a bialgebra. We recall that two coalgebras $C$ and $D$ are said to be {\em Morita-Takeuchi equivalent} if $\Mod^C$ and $\Mod^D$ are equivalent as linear categories.
Accordingly, we say that two left $H$-module coalgebras $C$ and $D$ are {\em $H$-Morita-Takeuchi equivalent} if $\Mod^C$ and $\Mod^D$ are equivalent as left $\Mod^H$-module categories. This equivalence relation is characterized as follows:

\begin{theorem}
  \label{thm:H-Morita-Takeuchi}
  Let $H$ be a bialgebra.
  For two left $H$-module coalgebras $C$ and $D$, the following are equivalent:
  \begin{enumerate}
  \item $C$ and $D$ are $H$-Morita-Takeuchi equivalent.
  \item ${}^C\Mod$ and ${}^D\Mod$ are equivalent as left ${}^H\Mod$-module categories.
  \item ${}_H\Mod^C$ and ${}_H\Mod^D$ are equivalent as left ${}_H\Mod$-module categories.
  \item ${}^C_H\Mod$ and ${}^D_H\Mod$ are equivalent as right ${}_H\Mod$-module categories.
  \item There are objects $M \in {}^C_H\Mod^D$ and $N \in {}^D_H\Mod^C$ such that
    \begin{equation*}
      \text{$M \coten_D N \cong C$ in ${}^C_H\Mod^C$}
      \quad \text{and} \quad
      \text{$N \coten_C M \cong D$ in ${}^D_H\Mod^D$}.
    \end{equation*}
  \end{enumerate}
\end{theorem}
\begin{proof}
  By the definitions of respective bicategroies, we have:
  \begin{itemize}
  \item [(1)] $\Leftrightarrow$ $\Mod^C$ and $\Mod^D$ are equivalent in ${}^H\Mod\mbox{-}\mathbf{Mod}$,
  \item [(3)] $\Leftrightarrow$ ${}_H\Mod^C$ and ${}_H\Mod^D$ are equivalent in ${}_H\Mod\mbox{-}\mathbf{Mod}$, and
  \item [(5)] $\Leftrightarrow$ $C$ and $D$ are equivalent in $H\mbox{-}\mathbf{Coalg}$.
  \end{itemize}
  Thus (1), (3) and (5) are equivalent by Theorems~\ref{thm:biequiv-Coalg} and \ref{thm:duality-1}.
  By applying the same argument to $H^{\cop}$, $C^{\cop}$ and $D^{\cop}$ (where $(-)^{\cop}$ means to take the opposite coalgebra), we see that (2), (4) and (5) are equivalent. The proof is done.
\end{proof}

We assume that the bialgebra $H$ is finite-dimensional.
Then the dual space $H^*$ of $H$ is also a bialgebra. There is a left $H^*$-module coalgebra $H^* \# C$, called the {\em cosmash product}, such that the category $\Mod^{H^* \# C}$ is isomorphic to ${}_H\Mod^C$ as left module categories over ${}_H\Mod$ ($\cong \Mod^{H^*}$). Specifically, $H^* \# C$ is the free left $H^*$-module $H^* \otimes C$ endowed with the comultiplication
\begin{equation*}
  \Delta_{H^* \# C}(f \otimes c) =
  (f_{(1)} \otimes (c_{(1)})_{[0]}) \otimes (f_{(2)} \star (c_{(1)})_{[1]} \otimes c_{(2)})
  \quad (f \in H^*, c \in C),
\end{equation*}
where $\star$ is the convolution product of $H^*$ and $c \mapsto c_{[0]} \otimes c_{[1]}$ is the right $H^*$-coaction induced by the left $H$-action of $H$.
The previous theorem implies:

\begin{theorem}
  Let $H$ be a finite-dimensional bialgebra.
  Two left $H$-module coalgebras $C$ and $D$ are $H$-Morita-Takeuchi equivalent if and only if the left $H^*$-module coalgebras $H^* \# C$ and $H^* \# D$ are $H^*$-Morita-Takeuchi equivalent.
\end{theorem}

\section{Module full subcategories}
\label{sec:applications-mod-full-subcat}

Let $\mathcal{M}$ be a Grothendieck category. A {\em closed subcategory} of $\mathcal{M}$ is a full subcategory of $\mathcal{M}$ that is closed under subobjects, quotient objects and direct sums. According to \cite[Theorem 2.5.5]{MR1786197}, there is a bijection between the set of subcoalgebras of a coalgebra $C$ and the set of closed subcategories of $\Mod^C$. The aim of this section is to extend this fact to the setting of module categories. Let $H$ be a bialgebra, and let $C$ be a left $H$-module coalgebra. An {\em $H$-module subcoalgebra of $C$} is a subcoalgebra of $C$ as well as an $H$-submodule. A {\em closed $\Mod^H$-module subcategory} $\mathcal{M}$ of $\Mod^C$ is a closed subcategory of $\Mod^C$ such that $X \otimes M \in \mathcal{M}$ for any $X \in \Mod^H$ and $M \in \mathcal{M}$. The final result of this paper is stated as follows:

\begin{theorem}
  \label{thm:closed-module-subcats}
  Let $H$ and $C$ be as above. Then there is a bijection between the set of $H$-module subcoalgebras of $C$ and the set of closed $\Mod^H$-module subcategories of $\Mod^C$.
\end{theorem}

To prove this theorem, we first recall how \cite[Theorem 2.5.5]{MR1786197} is established with emphasis on the use of Takeuchi's equivalence. We first remark:

\begin{lemma}
  Let $C$ and $D$ be coalgebras, and set $\mathcal{L} = \mathscr{L}(\Mod^C, \Mod^D)$. A morphism $i : F \to G$ in $\mathcal{L}$ is a monomorphism in $\mathcal{L}$ if and only if the component $i_M : F(M) \to G(M)$ is a monomorphism in $\Mod^C$ for all objects $M \in \Mod^C$.
\end{lemma}
\begin{proof}
  In view of Takeuchi's equivalence $\mathcal{L} \approx {}^C\Mod^D$, we may assume that $F = (-) \coten_C X$ and $G = (-) \coten_C Y$ for some $X, Y \in {}^C\Mod^D$ and $i$ is induced by a morphism $j : X \to Y$ in ${}^C\Mod^D$. We note that $i$ is a monomorphism in $\mathcal{L}$ if and only if $j$ is a monomorphism in ${}^C\Mod^D$.

  Suppose that $i$ is a monomorphism in $\mathcal{L}$. Then, since the cotensor product preserves kernels, $i_M = M \coten_C j : M \coten_C X \to M \coten_C Y$ is monic. Thus we have proved the `only if' part. The `if' part is clear. The proof is done.
\end{proof}

Given a subcoalgebra $D$ of $C$, we can identify $\Mod^D$ as a full subcategory
\begin{equation*}
  \Mod^D = \{ M \in \Mod^C \mid \coactr_M(M) \subset M \otimes D \},
\end{equation*}
which is closed. For the converse correspondence, we define
\begin{equation}
  \label{eq:tNM}
  t_{\mathcal{N}}(M) = (\text{the sum of all subcomodules of $M$ belonging to $\mathcal{N}$})
\end{equation}
for $M \in \Mod^C$ and a closed subcategory $\mathcal{N}$ of $\Mod^C$. Since $\mathcal{N}$ is closed under quotient objects and direct sums, $t_{\mathcal{N}}(M)$ belongs to $\mathcal{N}$. For a morphism $f: M \to M'$ in $\Mod^C$, one can define $t_{\mathcal{N}}(f) : t_{\mathcal{N}}(M) \to t_{\mathcal{N}}(M')$ to be the restriction of $f$ again since $\mathcal{N}$ is closed under quotients. Hence we have a functor $t_{\mathcal{N}} : \Mod^C \to \Mod^C$, whose image is equal to $\mathcal{N}$. It is obvious that $t_{\mathcal{N}}$ preserves direct sums. Moreover, since $\mathcal{N}$ is closed under subobjects, one can show that $t_{\mathcal{N}}$ preserves kernels. The argument so far shows that $t_{\mathcal{N}}$ is a subobject of the identity functor in the category $\mathscr{L}(\Mod^C, \Mod^C)$. By Takeuchi's equivalence, those subobjects are in bijection with subobjects of $C$ in ${}^C\Mod^C$, that is, subcoalgebras of $C$. The subcoalgebra corresponding to $t_{\mathcal{N}}$ is what we wanted.

\newcommand{\Sub}{\mathrm{Sub}}

From now on, we fix a bialgebra $H$ and a left $H$-comodule coalgebra $C$.
Theorem~\ref{thm:closed-module-subcats} will be proved by using the equivariant version of Takeuchi's equivalence.
Given an object $X$ of a category $\mathcal{A}$, we denote by $\Sub_{\mathcal{A}}(X)$ the class of subobjects of $X$. The set of $H$-module subcoalgebras of $C$ is nothing but $\Sub_{{}^C_H\Mod^C}(C)$. We write $\mathcal{L} = \mathscr{L}(\Mod^C, \Mod^C)$ and ${}_H\mathcal{L} = \mathscr{L}_{\Mod^H}^{\lax}(\Mod^C, \Mod^C)$. We note\footnote{Lemma~\ref{lem:nontrivial} is not trivial. Indeed, the same does not hold for the forgetful functor from the category of comodules over a coalgebra over a commutative ring $k$ to the category of $k$-modules \cite[Example 13]{W75}.}:

\begin{lemma}
  \label{lem:nontrivial}
  Let $F$ be an object of ${}_H\mathcal{L}$. Then the forgetful functor ${}_H\mathcal{L} \to \mathcal{L}$ induces an injective map from $\Sub_{{}_H\mathcal{L}}(F)$ to $\Sub_{\mathcal{L}}(F)$.
\end{lemma}
\begin{proof}
  ${}_H\mathcal{L}$, $\mathcal{L}$ and the forgetful functor ${}_H\mathcal{L} \to \mathcal{L}$ correspond to ${}^C_H\Mod^C$, ${}^C\Mod^C$ and the functor forgetting the left $H$-action, respectively, via the equivariant version of Takeuchi's equivalence.
  The proof is easy from the bicomodule side.
\end{proof}

We now identify $\Sub_{{}_H\mathcal{L}}(\id)$ as a subset of $\Sub_{\mathcal{L}}(\id)$ by this lemma. A subobject of $C$ in ${}^C_H\Mod^C$ is nothing but an $H$-module subcoalgebra of $C$, and thus they are in bijection with $\Sub_{{}_H\mathcal{L}}(\id)$. Thus, to prove Theorem~\ref{thm:closed-module-subcats}, it is sufficient to verify that $\tau \in \Sub_{\mathcal{L}}(\id)$ belongs to $\Sub_{{}_H\mathcal{L}}(\id)$ if and only if the closed subcategory of $\Mod^C$ corresponding to $\tau$ is closed under the action of $\Mod^H$.

\begin{proof}[Proof of Theorem \ref{thm:closed-module-subcats}]
  We choose an element $\tau$ of $\Sub_{\mathcal{L}}(\id)$ and let $D$ be the subcoalgebra of $C$ corresponding to $\tau$.
  We suppose that $\tau$ belongs to $\Sub_{{}_H\mathcal{L}}(\id)$. Then, by the equivariant version of Takeuchi's equivalence, $D$ is an $H$-module subcoalgebra of $C$. Thus, for $m \in M \in \Mod^D$ and $x \in X \in \Mod^H$, we have
  \begin{equation*}
    \coactr_{X \otimes M}(x \otimes m) = x_{(0)} \otimes m_{(0)} \otimes x_{(1)} m_{(1)}
    \in (X \otimes M) \otimes D,
  \end{equation*}
  which means $X \otimes M \in \Mod^D$. Namely, $\Mod^D$ is closed under the action of $\Mod^H$.

  We now suppose that $\Mod^D$ is closed under the action of $\Mod^H$. We recall that $\tau = t_{\mathcal{N}}$ is given from $\Mod^D$ by the formula \eqref{eq:tNM} with $\mathcal{N} = \Mod^D$. For $X \in \Mod^H$ and $M \in \Mod^C$, we have $X \otimes \tau(M) \subset X \otimes M$ by definition. Since $\tau(M)$ belongs to $\mathcal{N}$, the left hand side is contained in $\tau(X \otimes M)$. Thus we have an inclusion map
  \begin{equation}
    \label{eq:0}
    X \otimes \tau(M) \to \tau(X \otimes M)
  \end{equation}
  for $X \in \Mod^H$ and $M \in \Mod^C$. It is obvious that this makes $\tau$ a subobject of $\id$ in the category ${}_H\mathcal{L}$.
\end{proof}

\begin{remark}
  By Theorem~\ref{thm:equivalence-iota}, the inclusion morphism \eqref{eq:0} is always bijective if $H$ is a Hopf algera. However, in general, it is not surjective. An example is given as follows: We take $H = \bfk[x]$ to be the polynomial algebra. Then $H$ is a bialgebra by the comultiplication determined by $\Delta(x) = x \otimes x$. We also take $C = H$ (as a coalgebra), $D = \bigoplus_{j > 0} \bfk x^j$ (which is a subcoalgebra of $C$) and $\mathcal{N} = \Mod^D$. Then $\tau = t_{\mathcal{N}}$ is not a strong $\Mod^H$-module functor. Indeed, for $M = \bfk 1_H$ and $X = \bfk x$, we have $X \otimes \tau(M) = 0$ and $\tau(X \otimes M) = \bfk x \ne 0$.
\end{remark}

\def\cprime{$'$}

\end{document}